\documentclass[11pt, final]{article}
\usepackage[utf8]{inputenc}
\usepackage[english]{babel}
\usepackage{csquotes}
\usepackage{comment}
\usepackage{amsmath}
\usepackage{amsthm}
\usepackage{amssymb}
\usepackage{commath}
\usepackage{derivative}
\usepackage[mathcal]{eucal}
\let\mathscr\mathcal
\usepackage{color}
\usepackage{tensor}
\usepackage{breqn}
\usepackage{hyperref}
\usepackage{tikz}
\usepackage{pgfplots}
\pgfplotsset{compat=1.18}

\usepackage[a4paper, total={6in, 9in}]{geometry}

\newcommand{\supnorm}[1]{\left\lVert #1 \right\rVert_{\mathbf{L}^{\infty}}}

\newcommand{\pnorm}[2]{\left\lVert #2 \right\rVert_{#1}}

\newcommand{\R}{\mathbb{R}}
\newcommand{\eps}{\varepsilon}
\newcommand{\sgn}{\mathop{\rm sgn}}

\newcommand{\Czero}{\mathbf{C}} 
\newcommand{\Ck}[1]{\mathbf{C}^{#1}} 
\newcommand{\Cc}[1]{\mathbf{C}_\mathbf{c}^{#1}}
 
\renewcommand{\L}[1]{\mathbf{L}^{#1}} 
\newcommand{\Lloc}[1]{\mathbf{L}_{\mathbf{loc}}^{#1}} 

\newcommand{\cF}{\mathcal{F}}

\newtheorem{theorem}{Theorem}[section]
\newtheorem{corollary}{Corollary}[theorem]
\newtheorem{lemma}[theorem]{Lemma}

\title{Failure of uniqueness for scalar conservation laws}
\author{
		Shyam Sundar Ghoshal\thanks{Centre for Applicable Mathematics, Tata Institute of Fundamental Research} \and
		Abraham Sylla\thanks{LAMFA CNRS UMR 7352, Université de Picardie Jules Verne} \and
		Parasuram Venkatesh\footnotemark[1]
}
\date{}

\begin{document}
\maketitle
\begin{abstract}
	In this article, we develop what are, to the best of our knowledge, the first \emph{negative} results for scalar conservation laws with finite speed of propagation. While the work of Gargyants, Goritsky, and Panov establishes non-uniqueness of unbounded solutions for scalar conservation laws with spatially homogeneous flux, heterogeneity allows entropy solutions to exhibit more pathological behaviours. In particular, even for smooth and bounded initial data, the entropy solution may not be unique in the full half plane. Furthermore, the multiple entropy solutions all satisfy a finite speed of propagation property, so the non-uniqueness phenomena we demonstrate can arise in compact space-time domains. We begin with explicit examples where bounded initial data leads to $\L{\infty}$ blow-up despite flux regularity. More strikingly, we demonstrate that Kružkov's entropy inequalities alone fail to ensure uniqueness in this regime by constructing infinitely many entropy solutions to a single Cauchy problem with bounded initial datum, each continuous in time with respect to the $\L{1}$ norm. Thus, we demonstrate that the $\L{\infty}$ assumption is essential for the doubling of variables argument, and hence for the uniqueness of entropy solutions to scalar conservation laws. On the positive side, we develop a novel theory for scalar conservation laws with spatial heterogeneity by adapting the front tracking method. We recover uniqueness by imposing a Lax-type condition in addition to the entropy inequality, motivated by the properties of our front tracking approximations. Unbounded Kružkov solutions do not necessarily satisfy the weak formulation; we show that global weak solutions may not even \emph{exist} in a natural class for some Cauchy problems of this form, even when Kružkov entropy solutions exist. Finally, we detail examples demonstrating the sharpness of our assumptions and construct an explicit example of global \emph{ill-posedness} with bounded initial datum.
	
	\smallskip
	
	
	\smallskip
	
\end{abstract}
\tableofcontents

\section{Introduction}
Consider the Cauchy problem for the scalar conservation law in several spatial dimensions,
\begin{equation}
	\label{clawmulti}
	\begin{aligned}
		u_t+\sum_{i=1}^{n}f_i(x,u)_{x_i} & =0, \\
		u(x, 0) & = u_0(x).
	\end{aligned}
\end{equation}
Equations of this form appear in many physical contexts such as hydrodynamic models of traffic flow \cite{trafficnetwork}, sedimentation problems \cite{BKRT2004}, and multi-phase flows in porous media \cite{Jaffr1995}. Solutions of \eqref{claw} generally develop discontinuities in finite time, even for smooth initial data, and thus weaker notions of solution are required. In order to develop our explicit constructions for \eqref{clawmulti}, we work with the one-dimensional scalar conservation law
\begin{equation}
	\label{claw}
	\begin{aligned}
		u_t+f(x,u)_{x} & =0, \\
		u(x, 0) & = u_0(x).
	\end{aligned}
\end{equation}
The first comprehensive well-posedness theory for conservation laws with $u_0\in\L{\infty}$ was developed by Kružkov \cite{kruzkov}, following the work of Vol'pert \cite{Volpert} in the $\operatorname{BV}$ space. More generally, well-posedness was obtained also for scalar conservation laws in several space dimensions, a flux that depends on time, and source terms. Kružkov proved the existence of a unique `entropy' solution to \eqref{clawmulti}; when $n=1$ this takes the following form: $u$ is said to satisfy the entropy inequality if, for all non-negative test functions $\varphi\in\Cc{\infty}(\R\times[0,\infty))$ and $k\in\R$:
\begin{equation}\label{genent}
	\begin{split}
		&\iint_{\R^2_+}\abs{u(x,t)-k}\varphi_t(x,t)+\sgn(u(x,t)-k)\left[f(x,u(x,t))-f(x,k)\right]\varphi_x(x,t)dxdt \\
		\geq&\iint_{\R^2_+}\sgn(u(x,t)-k)f_x(x,k) \varphi(x,t) dxdt-\int_{\R}\abs{u_0(x)-k}\varphi(x,0)dx,
	\end{split}
\end{equation}
where $\R^2_+$ denotes the upper-half plane $\R\times[0,\infty)$ and $\operatorname{sgn}(\cdot)$ is the sign function, which evaluates to $1$ for non-negative values and $-1$ otherwise.

The novelty of this work is our analysis of the finer technical assumptions on the flux made by Kružkov. To the best of our knowledge, they constitute the first \emph{negative} results for scalar conservation laws. Our examples demonstrate that the structure of $f$ can play a decisive role in determining the well-posedness of \eqref{clawmulti}.

For equations of the form \eqref{claw}, a crucial assumption used for the a priori estimates was a lower bound on the mixed partial derivative, i.e. for some $L>0$ and for all $x,u\in\R$,
\begin{align}\label{mix}
	f_{xu}(x,u)\geq-L.
\end{align}
In several space dimensions, the left-hand side of \eqref{mix} is replaced by the spatial divergence of $f_u$; note that the two definitions coincide when $n=1$. However, this condition is not sharp. The equation \eqref{claw} has also been studied under the coercive genuinely non-linear assumptions in \cite{ConLawHJB}, relying on the correspondence between scalar conservation laws and Hamilton-Jacobi equations in one space dimension. For instance, if $f(x,u)=(2+\sin(x))u^2$, then $f$ does not satisfy \eqref{mix}, but the associated Cauchy problem is still well-posed since $f$ is `genuinely non-linear', i.e. $f_{uu}$ is non-vanishing.

In this paper we construct a counter-example demonstrating that the condition cannot be arbitrarily relaxed. That is, we construct a smooth flux violating \eqref{mix} such that the Cauchy problem may be ill-posed globally in time, even with smooth and bounded initial data. This indicates that the existence of solutions is sensitive to the structure of the flux, even with strong regularity assumptions.

The study of scalar conservation laws has uncovered significant pathologies when solutions transcend the classical class of bounded measurable functions. Early foundational work by Goritsky \cite{CountableUnB} introduced the construction of unbounded entropy solutions containing a countable number of sign-alternating shock waves, a result that directly facilitated the demonstration of non-uniqueness within the class of locally bounded functions $\L{\infty}_{\mathrm{loc}}$ for homogeneous equations. Such results were even obtained for zero initial data for fluxes satisfying a power-like growth such as $f(u)=u^3$ in \cite{ZeroInitU3}, and in the more detailed work \cite{QuasiLinUnB}. To define the boundaries of well-posedness, Panov \cite{QuasiLinMultiDUnB} introduced the $B_{\Phi}$ correctness classes, which preserve uniqueness by imposing growth restrictions on the solution proportional to the flux function's power growth. Gargyants \cite{ExpInitUnB} subsequently generalized the sign-alternating construction to exponential initial data, demonstrating that such solutions exhibit one-sided periodic behaviour and remain spatially bounded for any fixed positive time. Recent contributions by Gargyants \cite{NonExistence} further generalized these solution structures to exponential initial data, uncovering one-sided periodic behaviours and proving the non-existence of positive generalized entropy solutions for specific unbounded positive data. The systematic construction of these complex, discontinuous solutions has been further refined by Gargyants, Goritsky, and Panov using a geometric approach based on the Legendre transform \cite{LegendreUnB}.

Thus, non-uniqueness of unbounded solutions satisfying the entropy inequality is known, but they typically involve shock waves `from infinity' at the initial time, i.e. solutions have an infinite speed of propagation. The novelty of our approach is in demonstrating that such pathologies may arise even on compact space-time domains. In particular, our solutions are not merely globally unbounded, but fail to be in $\L{\infty}$ even on compact domains. Furthermore, such pathologies may be seen to arise naturally from $\Ck{\infty}_c$ initial data.

More generally, the non-uniqueness of entropy solutions to non-linear partial differential equations is an active research program. Most notably for the equations of fluid dynamics, the `convex integration' method pioneered by De Lellis and Székelyhidi \cite{convexintog} has proven fruitful in the construction of multiple admissible solutions to Cauchy problems with irregular initial data. For the compressible Euler equations, this regularity assumption can be relaxed, and the non-uniqueness is exhibited after the formation of singularities \cite{LipEul}. Our constructions follow a similar pattern; in particular even for smooth initial data, there may be infinitely many global-in-time entropy solutions to the Cauchy problem. A secondary active question is the search for a criterion that picks out a unique solution among the pathological ones. We provide such a criterion in the form of a Lax-type condition on interfaces generating the pathologies, returning a unique solution as the limit of front-tracking approximations.

Kružkov's proof for the uniqueness of entropy solutions does not require the assumption \eqref{mix}, it merely requires that the flux be continuously differentiable in all its variables with $f_x$ Lipschitz in $u$ on compact sets \cite[Theorem 1]{kruzkov}. However, the comparison principle used in the argument crucially relies on uniform bounds of the solutions being compared. Thus, it only demonstrates the uniqueness of $\L{\infty}$ solutions. We construct examples of fluxes satisfying all assumptions of the uniqueness theorem but without global $\L{\infty}$ a priori estimates. That is, the $\L{\infty}$ norm of entropy solutions may blow up in finite time, and the comparison principle is valid only up to this time. The blow-up goes hand-in-hand with non-uniqueness of entropy solutions; we construct explicit examples of Cauchy problems that admit infinitely many entropy solutions. The initial data may also be chosen such that $u_0\in\L{\infty}$. Although we work with an explicit example for illustration, our construction is generic, and applies to a large number of fluxes.

Our constructions are essentially one-dimensional, but can be carried out in arbitrary spatial dimensions by a simple extension. Thus, we direct our attention to fluxes of the specific form
\begin{align}\label{mult_form}
	f(x,u)=g(x)h(u),
\end{align}
where $g$ is locally Lipschitz and $h\in \Ck{2}(\R)$ is strictly convex with $h^{\prime}(0)=h(0)=0$, and thus $h(u)\geq0$ with equality holding if and only if $u=0$. In the sequel, we make a slight abuse of notation by writing $f(x,u(x,t))$ as $f\circ u(x,t)$. The precise structural assumptions are outlined in section~\ref{ass}. If $g$ in \eqref{mult_form} takes on both positive and negative values, then the flux $f$ is not genuinely non-linear everywhere and no known well-posedness theory applies. Thus, we develop a new framework based on the earlier work \cite{hetft}, inspired by the front tracking method of Dafermos \cite{polygon}.

Cauchy problems with such fluxes have entropy solutions that are obtained as limits of front tracking approximations. However, entropy solutions in the sense of \eqref{genent} may not be unique. We introduce a further constraint to recover uniqueness -- in particular, it is satisfied by limits of front tracking approximations, which therefore converge to a unique limit. In the final section, we show that outside the class of fluxes with form \eqref{mult_form}, the pathologies can worsen, such that global entropy solutions may not even exist.

Fluxes of this form satisfy neither the Kruzkov assumptions \cite{kruzkov} nor the coercivity conditions of Colombo, Perrollaz, and Sylla in \cite{ConLawHJB}. This has important consequences for well-posedness. In particular, although entropy solutions are unique, they may not exist as bounded functions, though our structural constraints outlined in sections~\ref{ass} can ensure that they are at least locally integrable.

Time-regularity of entropy solutions for multiplicative flux equations was analysed by Felix Otto in \cite{otto}, though well-posedness was not discussed. More generally, the study of conservation laws with spatial heterogeneity has been conducted in various forms. Motivated by relativistic applications, LeFloch and Ben-Artzi \cite{hypconlawmanifolds} analysed conservation laws on Riemannian or Lorentzian manifolds, where spatial heterogeneities naturally arise. Lengeler and Müller \cite{Lengeler} further proved the $\L{1}$ contraction property on closed manifolds. Convergence of a numerical scheme was demonstrated by Abraham Sylla \cite{sylla2024convergencefinitevolumescheme}, discretising the heterogeneity and applying the theory of discontinuous flux \cite{discofluxog}. Ghoshal and Venkatesh \cite{singleshock} demonstrated the asymptotic emergence and $\L{2}$ stability of simple shock for spatially heterogeneous conservation laws. Anne-Laure Dalibard \cite{dalibard} developed a pure $\L{1}$ theory for such conservation laws, based on the kinetic formulation \cite{perthame}. A front tracking approximation scheme for such equations was developed by Parasuram Venkatesh \cite{hetft}, which we leverage here to construct entropy solutions of the Cauchy problem \eqref{claw}.

\subsection{Structure of the paper}
The rest of this paper is structured as follows. In the following section~\ref{prelim} we collate some results for scalar conservation laws in one spatial dimension that are crucial for our constructions. In section~\ref{ass} we list the structural assumptions we impose on the flux \eqref{mult_form} in order to develop the well-posedness theory. Section~\ref{summary} colloquially summarises our results for the ease of reference. Next, in section~\ref{wp} we demonstrate the possibility of $\L{\infty}$ blow up in finite time, and show how Kružkov's entropy formulation can still make sense, though the entropic solution may not satisfy \eqref{claw} in the weak sense, unless it is uniformly bounded. We also prove the non-uniqueness of entropy solutions, first with an explicit example and then more generically in Section~\ref{generic nonuni}. In order to recover a unique solution, we develop an approximation theory for the Cauchy problem, through which we also prove the existence of solutions in the large. Section~\ref{ft approx} concerns the front tracking method we use to prove the existence of Kružkov entropy solutions; we also prove that the approximations converge to a unique limit by imposing a Lax-type condition that allows us to recover uniqueness. In section~\ref{non int}, we illustrate the construction of some explicit solutions for equations whose fluxes do not satisfy all the structural assumptions. Finally, in section~\ref{illp}, we construct a Cauchy problem for a scalar conservation law for which no global solution exists.

\subsection{Preliminaries}\label{prelim}
Here we recall some preliminary concepts that we will use in the sequel. Dafermos' classical text \cite{Dafermos2016} is a comprehensive reference for all the results here. Differentiable solutions of \eqref{claw} exist locally in time for smooth initial data \cite{Kato1975}, but generally break down in finite time. Hence the concept of entropy is introduced -- solutions satisfying a family of inequalities of the form
\begin{equation}\label{Het entropy}
	\eta(u)_t+Q(x,u)_x+\eta^{\prime}(u)f_x(x,u)-Q_x(x,u)\leq0,
\end{equation}
in the sense of distributions over $\Cc{\infty}(\mathbb{R}\times[0,\infty))$ for all pairs of functions $\eta,Q$ such that the `entropy' $\eta$ is a convex function and the `entropy flux' $Q(x,\cdot)$ is an antiderivative of $\eta^{\prime}(\cdot)f_u(x,\cdot)$ for each fixed $x$. However, this conception of entropy is too strong when solutions are unbounded, as we shall see later. The characteristic equations associated to \eqref{claw}, valid for Lipschitz solutions, are given by
\begin{equation}
	\label{char ode}
	\begin{aligned}
		\dot{q}(t)&=f_u(q(t),p(t)), \\
		\dot{p}(t)&=-f_x(q(t),p(t)).
	\end{aligned}
\end{equation}
If $u$ is a weak entropic solution of \eqref{claw} which is continuous on either side of a curve $\gamma(t)$, then $\gamma$ satisfies the Rankine-Hugoniot condition
\begin{equation}\label{RH}
	\dot{\gamma}(t)=\dfrac{f(\gamma(t),u(\gamma(t)+,t))-f(\gamma(t),u(\gamma(t)-,t))}{u(\gamma(t)+,t)-u(\gamma(t)-,t)}.
\end{equation}
Let us recall Dafermos' theory of generalised characteristics \cite{GenChar} for genuinely non-linear scalar conservation laws, which allows for the use of characteristics \eqref{char ode} even when the solution develops discontinuities \cite{Filippov}.

Given an entropy solution $u(x,t)$ and flux $f$ that is strictly convex we have that from every point $(x,t)$ with $t>0$, we can define a unique \textit{forward characteristic} $y_f:[t,\infty)\to\mathbb{R}$ and a non-empty set of \textit{backward characteristics} $y_b:[0,t]\to\mathbb{R}$, i.e. Lipschitz curves with $y(t)=x$ solving the differential inclusion
\begin{equation}\label{DafDiff}
	\dot{y}(s)\in[f_u(y(s),u(y(s)+,s)),f_u(y(s),u(y(s)-,s))]
\end{equation}
on their respective domains, where $u(x\pm,t)$ respectively denote the left and right traces in space of $u$ at $(x,t)$. If the flux is \emph{uniformly} convex in the conserved variable, then these traces always exist for positive times due to the regularising effect of the non-linearity \cite{robyr}. Furthermore, the interval in \eqref{DafDiff} is non-empty for all $(x,t)$ with $t>0$ since $f_u$ is monotone increasing in the second argument and the traces satisfy the Oleinik inequality $u(x-,t)\geq u(x+,t)$ at positive times. Note that along the trajectories of \eqref{char ode},
\begin{equation}\label{f conserved}
	\begin{split}
		\frac{d}{dt}f(y(t),z(t))&=f_x(y(t),z(t))\dot{y}(t)+f_u(y(t),z(t))\dot{z}(t) \\
		&=f_x(y(t),z(t))f_u(y(t),z(t))-f_u(y(t),z(t))f_x(y(t),z(t)) \\
		&=0.
	\end{split}
\end{equation}
The differential inclusion is more restrictive than it appears prima facie. If $y:[t_0,T]\to\mathbb{R}$ is a Lipschitz solution to \eqref{DafDiff} for some $t_0\geq0$, then for almost all $t\in[t_0,T]$ the curve either satisfies the differential \emph{equation} or the Rankine-Hugoniot condition, i.e. we have that
\[
\dot{y}(t)=
\begin{cases}
	&f_u(y(t),u(y(t),t))\text{ if }u(y(t)-,t)=u(y(t)+,t), \\[1ex]
	\smallskip
	&\dfrac{f(y(t),u(y(t)-,t))-f(y(t),u(y(t)+,t))}{u(y(t)-,t)-u(y(t)+,t)}\text{ if }u(y(t)-,t)>u(y(t)+,t).
\end{cases}
\]
The extremal backward characteristics from any point $(x,t)$ with $t>0$ are `genuine' , i.e. $y(s)$ for $s\in(0,t)$ are points of continuity for $u$, and $u(y(s),s)=z(s)$.

With these generalised characteristics, a front tracking scheme can be constructed to approximate solutions of Cauchy problems with respect to certain conservation laws with heterogeneous flux \cite{hetft}. The basic idea is to approximate $\operatorname{sgn}(u_0(x))f(x,u_0(x))$ by a piece-wise constant function, and solve the `generalised Riemann problem' at discontinuities by either a Rankine-Hugoniot shock curve or an approximate rarefaction fan, and solving forward by entropic shocks whenever fronts interact.

\subsection{Structural assumptions}\label{ass}
Let us explicitly state our assumptions on the multiplicative \eqref{mult_form} flux $f(x,u)=g(x)h(u)$.
\begin{flalign}
	\label{R}\tag{R}&\mbox{\textbf{Regularity: }}h\in\Ck{2}(\mathbb{R}),g\text{ is locally Lipschitz in }\R.&& \\
	\notag&\text{In addition, if }a,b\in\R\text{ such that for all }x\in(a,b), g(x)\neq0,\text{ then }g\in\Ck{2}([a,b]). \\
	\label{S}\tag{S}&\mbox{\textbf{Stationarity: }}h(0)=h^{\prime}(0)\equiv0. \\
	\label{SC}\tag{SC}&\mbox{\textbf{Strict Convexity: }}h^{\prime}\text{ is a strictly increasing function}. \\
	\label{LL}\tag{LL}&\mbox{\textbf{Log Lipschitz Growth:} }\log\circ h\in\operatorname{Lip}((-1,1)^c).
\end{flalign}
Throughout this paper, we will consider convexity in terms of \eqref{SC}, and mutatis mutandis for concavity. Note that \eqref{LL} ensure the finite speed of propagation property on compact sets in space for merely flux-bounded solutions: if $f\circ u(x,t)$ is uniformly bounded, then so is $f_{u}\circ u(x,t)$ even if $u$ is unbounded, since for non-zero $u$ we have that
\[
f_u\circ u(x,t)=g(x)h^{\prime}(u)=g(x)h(u)\dfrac{h^{\prime}(u)}{h(u)}.
\]
Note that $h(u)$ in the denominator is non-zero if $u\neq0$ by \eqref{S} and \eqref{SC}. However, finite speed merely on compact sets is not sufficient for well-posedness, as we demonstrate in section~\ref{infspeed}. In addition to these, we need assumptions to control how the flux behaves when $g$ vanishes. Let $G=\{x\in\mathbb{R}:g(x)=0\}$. In order to construct entropy solutions, we require that $G$ consist of disjoint (possibly degenerate) closed intervals $[a,b]$ (possibly with $a=-\infty$ or $b=+\infty$) such that for each such interval, $G\backslash[a,b]$ is still closed, i.e. the zeroes of $g$ can only accumulate on a complete interval where $g$ vanishes. We can make this a precise assumption as follows: the set of zeroes $G$ of $g$ must satisfy a discreteness condition.
\begin{flalign}
	\label{B}\tag{B}&\mbox{\textbf{Boundary Regularity of G: }}\text{ the boundary }\partial G\text{ has no accumulation point.}&&
\end{flalign}
We remark that the regularity assumption \eqref{R} does not merely state that g is a $\Ck{2}$ function whenever it is non-zero, but rather that it is $\Ck{2}$ on all \emph{closed} intervals that $g$ is non-zero in the interior of. That is to say, we assume that the regularity of $g$ extends up to the boundary of every open interval contained in $G^c$.

At points of $\partial G$, we require that $g$ have a finite order of vanishing, that is, if $\overline{x}\in\partial G$, then for some positive exponent $\eta\geq1$ and constant $K$, we must have that $\abs{g(x)}\geq K\abs{x-\overline{x}}^\eta$ in some neighbourhood of $\overline{x}$. More precisely,
\begin{flalign}
	\label{V}\tag{V}\mbox{\textbf{Order of Vanishing: }}&\exists \eta\geq1,K>0:\forall \overline{x}\in\partial G,\exists\delta>0:\abs{g(x)}\geq K\abs{x-\overline{x}}^\eta,&& \\
	&\notag x\in(\overline{x}-\delta,\overline{x}+\delta)\cap G^c.
\end{flalign}
Let $\eta\geq1$ be the infimum over all exponents larger than $1$ that suffice for all points of $\partial G$, then we impose the following growth assumption on $h$.
\begin{flalign}
	\label{CG}\tag{CG}\mbox{\textbf{Compensating Growth: }}\exists C,\varepsilon,M>0:h(u)\geq C\abs{u}^{\eta+\varepsilon}\text{ for }\abs{u}>M.&&
\end{flalign}
The assumption \eqref{CG} ensures that $u\in\Lloc{1}(\R\times[0,\infty))$ whenever $f\circ u$ is uniformly bounded and $g$ is non-zero. Of course, if $g=0$, then $f\circ u=0$ for any measurable function $u$. We remark that when $G$ is empty, we have $\eta=1$ vacuously, and the assumption \eqref{CG} essentially amounts to uniformly superlinear growth of the flux. This is slightly stronger than what is required for well-posedness of \eqref{claw} with strictly convex/concave flux.

\subsection{Summary of results}\label{summary}
Large parts of the paper involve explicit constructions of counter-examples. For the readers' convenience, we colloquially summarise our findings here. We remark that our negative results hold in arbitrary spatial dimensions by trivial extension, while the positive results are limited to one spatial dimension.
\begin{enumerate}
	\item $\L{\infty}$ blow-up: in section~\ref{wp}, we demonstrate that solutions of \eqref{clawmulti} may not remain bounded globally in time. Theorem~\ref{l inf blowup} makes precise the conditions under which such blow-up generically occurs.
	\item Non-uniqueness: following the possibility of unbounded entropy solutions to \eqref{clawmulti}, in section~\ref{nonuni} we demonstrate that they may not be unique. Thus, we show that Kruzkov's `doubling-of-variables' argument crucially relies on $\L{\infty}$ bounds and fails without them.
	\item Well-posedness: in section~\ref{ft approx}, we show that the equation \eqref{claw} is well-posed for fluxes satisfying the structural assumptions outlined in section~\ref{ass}, i.e. entropy solutions exist for locally integrable initial data, and we recover uniqueness and stability despite the negative results of section~\ref{nonuni} by imposing a further Lax-type condition in addition to the entropy inequality. In section~\ref{infspeed} we demonstrate how infinite speed of propagation allows for a different kind of non-uniqueness with a shock emerging from infinity. This construction is similar in spirit to the ones developed by Gargyants, Goritsky, and Panov \cite{LegendreUnB}.
	\item Ill-posedness: finally, we demonstrate that there exist conservation laws \eqref{clawmulti} for which \emph{no} global entropy solution may exist, by constructing an explicit example in section~\ref{illp} where the entropy solution blows up everywhere in space at a finite positive time.
\end{enumerate}
All together, our results show that the theory of scalar conservation laws with spatially heterogeneous flux is not a simple extension of the homogeneous theory. To the best of our knowledge, these are the first negative results for scalar conservation laws on compact sets; the examples of non-uniqueness as comprehensively catalogued in \cite{LegendreUnB} involves shocks emerging `from infinity' to disturb the canonical solutions.

\section{Well-posedness}\label{wp}
The notion of entropy solution for us will be an adaptation of the standard one used for scalar conservation laws, i.e. $u \in \Lloc{1}(\R \times[0,T))$ is a solution of the Cauchy problem \eqref{claw} in $\Omega_T=\R\times[0,T]$ if $f\circ u$ is uniformly bounded in $\Omega_T$, and for all $\varphi\in \Cc{\infty}(\R \times[0,T),[0,\infty))$ and $k\in\R$:
\begin{equation}\label{entweak}
	\begin{split}
		&\iint_{\R^2_+}\abs{u(x,t)-k}\varphi_t(x,t)+\sgn(u(x,t)-k)\left[f\circ u(x,t)-f(x,k)\right]\varphi_x(x,t)dxdt \\
		\geq&\iint_{\R^2_+}\sgn(u(x,t)-k)f_x(x,k) \varphi(x,t) dxdt-\int_{\R}\abs{u_0(x)-k}\varphi(x,0)dx.
	\end{split}
\end{equation}
Equivalently, we can characterise entropy solutions as those satisfying the above inequality for positive test functions supported in the open upper half plane and continuously approaching the initial value $u_0$ in $\Lloc{1}$ as $t\to0$. We remark that it is sufficient for $u$ and $(t, x) \mapsto f(x,u(t, x))$ to be merely $\Lloc{1}$ in order to make sense of this definition. However, entropy solutions may not satisfy the integral or `weak' formulation of \eqref{claw}, unless $u\in \L{\infty}(\R)$, in which case the integral formulation follows directly from \eqref{entweak} by considering $k=\pm\supnorm{u}$. We also show that global weak solutions may not even exist in the natural class for some Cauchy problems.

Entropy solutions also satisfy the trace property in domains of genuine non-linearity, i.e. where $g\neq0$. In fact, the solutions are locally special functions of bounded variation at all but countably many time levels, following the argument in \cite{robyr}.
\begin{lemma}\label{trace}
	If $x\in G^c$ and $t>0$, then for any entropy solution $u$, the left and right spatial traces $u(x\pm,t)$ exist and satisfy the inequality $u(x-,t)\geq u(x+,t)$ if $g(x)>0$ and $u(x-,t)\leq u(x+,t)$ if $g(x)<0$.
\end{lemma}
\begin{proof}
	Let $t>0$, $x\in G^c$ and assume without loss of generality that $g(x)>0$. Since $G$ is a closed set, there is a maximal open interval $(a,b)\subseteq G^c$ such that $x\in(a,b)$. There exists some $d>0$ be such that $[x-d,x+d]\subset(a,b)$. Set $M=\supnorm{f\circ u}$, then by \eqref{LL}, we have that $f_u\circ u$ is also uniformly bounded, say by $R>0$. Now, there exists some $\varepsilon>0$ small enough such that $t>\varepsilon$ and $d>R\varepsilon$. Consider the trapezoidal domain
	\[
	T_x=\{(y,s):s\in[t-\varepsilon,t],y\in[x-d+R(s-t+\varepsilon),x+d-R(s-t+\varepsilon)]\}.
	\]
	Since $f\circ u$ is bounded, and $[x-d,x+d]$ is a compact subset of $G^c$, we have that $u$ is uniformly bounded in $T_x$ as well. If we let $u_0(x)=u(x,t-\varepsilon)$ for $x\in[x-d,x+d]$ and zero otherwise, then in the domain $T_x$, the bounded measurable function $u$ is an entropy solution of a genuinely non-linear scalar conservation law. Since the entropy solutions of such equations correspond with viscosity solutions of the associated Hamilton-Jacobi equation \cite{correspondence}, it follows that $u$ has spatial traces in the domain $T_x\cap[s>t-\varepsilon]$, which includes the point $(x,t)$. Since $(x,t)$ was arbitrarily chosen from $G^c\times(0,\infty)$, we are done.
\end{proof}
Note that if $u\in \L{\infty}$, then the traces of $f\circ u$ on $G\times[0,T]$ are zero. In the next section we demonstrate that working with $\Lloc{1}$ is necessary for multiplicative fluxes, as the $\L{\infty}$ norm can blow up even for uniformly bounded initial data. For such solutions, $f\circ u$ may have non-zero trace on the degenerate interface.

\subsection{$\L{\infty}$ blow-up of entropy solutions}
For particular fluxes, we can show explicitly that the Cauchy problem corresponding to \eqref{claw} is not globally well-posed in $\L{\infty}$, though it still is in $\Lloc{1}$. Let $f$ be given by
\begin{equation}
	\label{flipping flux}
	f(x,u)=xu^2.
\end{equation}
Although $f$ as defined satisfies the assumptions of Theorem~\ref{vander} proved later, it is instructive to first visualise the blow-up phenomenon explicitly. This flux is convex for $x>0$ and concave for $x<0$, but degenerate for $x=0$. The characteristic equations thus reduce to
\begin{align}
	\dot{q}(t)&=2q(t)p(t), \label{char ode flipping flux q} \\
	\dot{p}(t)&=-p(t)^2. \label{char ode flipping flux p}
\end{align}
We can integrate this system by first solving in $p$, and then using the Hamiltonian structure of the system. More precisely, if $p(0) = 0$, then, $p \equiv 0$. Otherwise, 
\[
p(t)=\dfrac{1}{t+p(0)^{-1}}.
\]
We see that $p$ blows up in finite time if $p(0) < 0$. Then, using the fact that $t \mapsto f(q(t), p(t))$ is constant, we deduce that 
\begin{equation}
	\label{q formula 1}
	q(t) = q(0)\left(tp(0)+1\right)^2,
\end{equation}
which holds for as long as $p$ exists, i.e. for $t<-{p(0)}^{-1}$. If $q(0)=0$, however, the trivial solution satisfies \eqref{char ode flipping flux q} globally in time.

With this explicit formula for characteristics, we have the following blow-up result.
\begin{theorem}\label{l inf blowup}
	Fix $u_0\in \Czero \cap \L{\infty}(\R)$. Assume that $u_0$ is non-increasing on $(-\infty,0]$, non-decreasing on $[0,\infty)$, and that $u_0(0) < 0$. Then, a continuous entropy solution of \eqref{claw} exists for $t<-u_0(0)^{-1}$, and the $\L{\infty}$ norm of the solution blows up at $(x,t)=(0,-u_0(0)^{-1})$.
\end{theorem}

\begin{proof}
	A continuous solution exists as long as distinct characteristics do not intersect. The characteristic emanating from $x=0$ is just the trivial one. Hence, it is sufficient to show that characteristics emanating from $x>0$ or $x<0$ do not intersect with each other, respectively. Let $0\leq q_1(0)<q_2(0)$; then by hypothesis $u_0(0)\leq p_1(0)\leq p_2(0)\leq0$. For $0<t<-u_0(0)^{-1}$, by \eqref{q formula 1} we have that
	\begin{align}
		q_1(t)&=q_1(0)\left(tp_1(0)+1\right)^2, \label{q1} \\
		q_2(t)&=q_2(0)\left(tp_2(0)+1\right)^2. \label{q2}
	\end{align}
	Now, we can divide \eqref{q1} by \eqref{q2} to get
	\[
	\dfrac{q_1(t)}{q_2(t)}=\dfrac{q_1(0)}{q_2(0)}\left(\dfrac{tp_1(0)+1}{tp_2(0)+1}\right)^2.
	\]
	Since the right-hand side is always less than one, we see that the characteristics, as long as they exist, cannot meet. The same argument applies, \textit{mutatis mutandis}, for characteristics emanating from negative values of $x$. Blow-up of the characteristics as $t$ approaches $-u_0(0)^{-1}$ when $u_0(0)<0$ follows from \eqref{char ode flipping flux p}, continuity, and positive/negative monotonicity.
\end{proof}
If $u_0$ is merely bounded and non-decreasing/non-increasing on $[0,\infty)$ and $(-\infty,0]$ respectively, similar results hold. At points of discontinuity, the forward characteristics fan out as a rarefaction wave, and are again non-intersecting as long as they exist. Consider
\begin{equation}\label{blowupdata}
	u_0(x)=
	\begin{cases}
		-1&\text{ if }\abs{x}\leq1, \\
		0&\text{ if }\abs{x}>1,
	\end{cases}
\end{equation}
which satisfies the prescribed monotonicity conditions. Then we have a blow-up at $t=1$, and the characteristics can be pictured as in Figure~\ref{figblowup}
\begin{figure}[htbp]
	\centering
	\begin{tikzpicture}[scale=3]
		% Define domain
		\def\xmin{-2}
		\def\xmax{2}
		\def\tmin{0}
		\def\tmax{1}
		
		% Draw axes
		\draw[<->] (\xmin-0.1,0) -- (\xmax+0.1,0) node[right] {$x$};
		\draw[->] (0,0) -- (0,\tmax+0.1) node[above] {$t$};
		
		% Add axis labels
		\foreach \x in {-2,-1,0,1,2}
		\draw (\x,0.05) -- (\x,-0.05) node[below] {$\x$};
		
		% Draw domain boundary
		\draw[black, dashed] (\xmax,\tmax) -- (\xmin,\tmax);
		
		% Label t = 1
		\node[right] at (\xmax, \tmax) {$t=1$};
		
		\foreach \xinit in {-2, -1.8, -1.6, -1.4, -1.2, -1, 1, 1.2, 1.4, 1.6, 1.8, 2}{
			\draw[black] (\xinit,0) -- (\xinit,\tmax);
		}
		
		% Non-trivial characteristics for |x| < 1
		\foreach \xinit in {-1,-0.8,-0.6,-0.4,-0.2,0,0.2,0.4,0.6,0.8,1} {
			\draw[black, thick] plot[domain=0:1, samples=100, variable=\t] ({\xinit*(pow(1-\t,2))}, \t);
		}
		
		% Rarefaction characteristics for x = 1
		\foreach \uinit in {0.1, 0.2, 0.3, 0.4, 0.6, 0.8} {
			\draw[black, thin] plot[domain=0:1, samples=100, variable=\t] ({(pow(1-\uinit*\t,2))}, \t);
			\draw[black, thin] plot[domain=0:1, samples=100, variable=\t] ({-(pow(1-\uinit*\t,2))}, \t);
		}
	\end{tikzpicture}
	\caption{Characteristics associated with the Cauchy problem for \eqref{claw} with $f(x,u)=xu^2$ and initial data given by \eqref{blowupdata}. The $\L{\infty}$ norm of the entropy solution $u$ blows up along $\{x=0,t\geq1\}$.}\label{figblowup}
\end{figure}
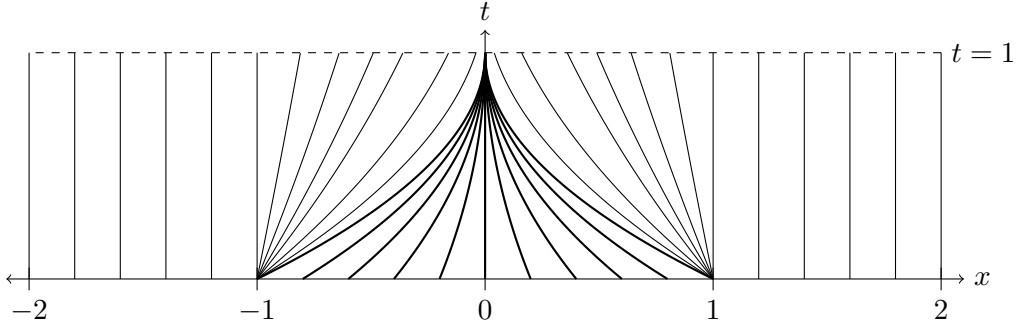

\subsection{Non-uniqueness}\label{nonuni}
Unexpectedly, the entropy condition \eqref{entweak} is not sufficient to ensure uniqueness. If a uniformly bounded entropy solution exists, then it is still unique \textit{in the class of} bounded solutions \cite[Theorem 1]{kruzkov}, but we know from Theorem~\ref{l inf blowup} that we cannot remain in this class. Let us consider the flux $f(x,u)=xu^2$ as before.

Recall that stationary solutions in $x>0$ are of the form $f(x,u(x))=c$ for $u(x)\leq0$ and some real number $c$, and that the positive stationary solution corresponding to $c$ does not satisfy \eqref{entweak}, to rephrase. However, if we add a dissipative term to the other side of the boundary, we can cancel out the trace term at the interface, and this `glued' solution stands in contrast with the one we would obtain by solving as a rarefaction fan.

\subsubsection{Entropy-preserving solution}
First, we construct a \textit{discontinuous} entropy-preserving stationary solution of \eqref{claw} for the flux $f(x,u)=xu^2$, that is a function that satisfies the entropy inequality as an equality. The construction generically works for any similarly multiplicative flux. We also remark that instead of $g(x)=x$, we could also take a function $g$ that is uniformly bounded but agrees with the identity function on some compact interval $[-K,K]$ where $K>0$. Our results are essentially localised to the interface $x=0$ where the flux flips from convex to concave.

\begin{lemma}\label{stat sol}
	Define 
	%The function $u(x,t)=u(x)$, where
	\[
	u(x)=
	\begin{cases}
		-1/\sqrt{-x}&\text{ if }x<0, \\
		1/\sqrt{x}&\text{ if }x>0.
	\end{cases}
	\]
	Then $u$ is stationary entropy solution of the equation \eqref{claw} with flux $f(x,u)=xu^2$; 
	%and initial data $u_0(x)=u(x)$ in the sense of \eqref{entweak}; 
	in particular, it satisfies the entropy inequality as an \emph{equality}.
\end{lemma}
\begin{proof}
	Since $u$ is time-independent, it is trivially $\Lloc{1}$-continuous in time. Thus, it is sufficient to prove the entropy inequality for positive test functions $\varphi$ supported in the open upper-half plane. Furthermore, $f(x,u(x))=\sgn(x)$ for all $x\neq0$, hence $u$ is a classical solution of \eqref{claw} away from the interface $\{x=0\}$. Now, let $\varphi\in \Cc{\infty}(\R\times(0,\infty);[0,\infty))$ and $k \in \R$. Choose some $\eps>0$ such that $\eps^{-1}>k^2$, so that $\abs{u(x)}>\abs{k}$ for all $\abs{x}\leq\eps$. Then, by splitting up the integral in \eqref{entweak} and integrating by parts, we have that
	\begin{align*}
		E:=& \iint_{\R^2_+}\abs{u(x)-k}\varphi_t+\sgn(u(x)-k)\left\{\left[f(x,u(x))-f(x,k)\right]\varphi_x-f_x(x,k)\varphi\right\}dxdt \\
		=&\int_0^{\infty}\int_{-\eps}^{\eps}\abs{u(x)-k}\varphi_t+\sgn(u-k)\left\{\left[f(x,u)-f(x,k)\right]\varphi_x-f_x(x,k)\varphi\right\}dxdt \\
		&-\int_{0}^{\infty}\left[f(\eps,u)-f(\eps,k)\right]\varphi(\eps,t)dt-\int_{0}^{\infty}\left[f(-\eps,u)-f(-\eps,k)\right]\varphi(-\eps,t)dt.
	\end{align*}
	Now take the limit as $\eps \to 0$. Using that $f \circ u$ is constant, the integrability of integrands, and the fact that $f(\varepsilon,k)\to0$ as $\varepsilon\to0$ for any fixed $k\in\R$, we obtain:
	\[
	E=\int_{0}^{\infty}(-f(1,u(1))-f(-1,u(-1)))\varphi(0,t)dt=0.
	\]
	Since $k,\varphi$ were arbitrary, this proves our claim that $u$ is an entropy solution that satisfies the inequality \eqref{entweak} as an equality.
\end{proof}
We remark that if we instead define $u(x)=0$ for $x>0$ while maintaining $u(x)=-1/\sqrt{-x}$ for $x<0$, the same calculation shows that this newly defined $u$ would still be a stationary entropy solution of \eqref{claw}, albeit one satisfying the entropy condition as a strict inequality. However, in both cases, $u$ does \emph{not} satisfy the integral formulation of \eqref{claw}, i.e. it is not a weak solution.

\subsubsection{Non-uniqueness with one discontinuity}
Before we establish non-uniqueness of entropy solutions of \eqref{claw} even for $\L{\infty}$ initial data, we first establish it in a special case involving discontinuities, i.e. shocks. We use this as a stepping stone for the final construction. For the same flux $f(x,u)=xu^2$ as before, consider the initial datum
\begin{equation}\label{shocking nonuniqueness}
	u_0(x)=
	\begin{cases}
		-1/\sqrt{-x}&\text{ if }x<0, \\
		0&\text{ if }x>0.
	\end{cases}
\end{equation}
A similar calculation as in the proof of Lemma~\ref{stat sol} shows that this datum is a stationary entropy solution of \eqref{claw}. 

Later, in Lemma~\ref{gen riem} we show this for more general stationary solutions. However, it is not the unique entropy solution of the Cauchy problem. Define
\begin{equation}\label{sqrt RH curve}
	u(x,t)=
	\begin{cases}
		-1/\sqrt{-x}&\text{ if }x<0, \\
		1/\sqrt{x}&\text{ if }0<x<t^2/4, \\
		0&\text{ if }x>t^2/4.
	\end{cases}
\end{equation}
We can also write this function succinctly as
\[
u(x,t)=
\begin{cases}
	\operatorname{sgn}(x)/\sqrt{\abs{x}}&\text{ if }x<y(t), \\
	0&\text{ otherwise},
\end{cases}
\]
where $y(t)=t^2/4$. This function is illustrated in Figure~\ref{disco}.

We claim that $u$ is also an entropy solution of the same Cauchy problem. For test functions $\varphi$ supported away from $t=0$, this is clear from Lemma~\ref{stat sol} and the fact that $u$ satisfies the Rankine-Hugoniot condition along the curve $\gamma(t)=t^2/4$ with a well-signed jump. Since $1/\sqrt{x}$ is also integrable near zero, it follows that $\pnorm{\L{1}}{u(\cdot,t)-u_0}\to0$ as $t\to0$, and thus $u$ is an entropy solution of the Cauchy problem. In \cite{timecont}, the authors construct a solution of Burgers equation that satisfies the entropy inequality in the upper half plane but fails to achieve the initial value as $t\to0$. In our example, however, we have $\Lloc{1}$-continuous non-unique solutions.

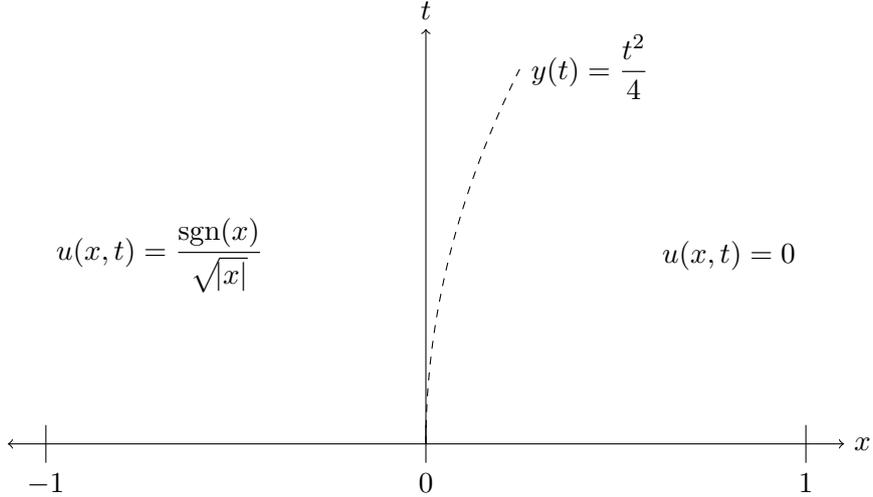
\begin{figure}[htbp]
	\centering
	\begin{tikzpicture}[scale=5]
		% Define domain
		\def\xmin{-1}
		\def\xmax{1}
		\def\tmin{0}
		\def\tmax{1}
		
		% Draw axes
		\draw[<->] (\xmin-0.1,0) -- (\xmax+0.1,0) node[right] {$x$};
		\draw[->] (0,0) -- (0,\tmax+0.1) node[above] {$t$};
		
		% Add axis labels
		\foreach \x in {-1,0,1}
		\draw (\x,0.05) -- (\x,-0.05) node[below] {$\x$};
		
		% Label y(t)
		\node[right] at (1/4,1) {$y(t)=\dfrac{t^2}{4}$};
		
		% Label values
		\node[right] at (-1,1/2) {$u(x,t)=\dfrac{\operatorname{sgn}(x)}{\sqrt{\abs{x}}}$};
		\node[left] at (1,1/2) {$u(x,t)=0$};
		
		% Blowup characteristic
		\draw[black, dashed] plot[domain=0:1, samples=100, variable=\t] ({pow(\t,2)/4}, \t);
	\end{tikzpicture}
	\caption{One of the infinitely many Kružkov entropy solutions to the Cauchy problem for \eqref{claw} with flux $f(x,u)=xu^2$ and initial data given by \eqref{shocking nonuniqueness}.}\label{disco}
\end{figure}

\subsubsection{Non-uniqueness from bounded initial data}
By Kruzkov's doubling-of-variables argument \cite[Theorem 1]{kruzkov}, we know that uniqueness holds for the Cauchy problem within the class of bounded entropy solutions whenever they exist. However, when the $\L{\infty}$ norm blows up for bounded initial data, we can exploit it to generate infinitely many entropy solutions on a time horizon that is large enough to allow for blow-up. Consider the initial datum
\begin{equation}\label{bounded nonuniqueness}
	u_0(x)=
	\begin{cases}
		-1/\sqrt{-x}&\text{ if }x<-1, \\
		0&\text{ if }x>-1.
	\end{cases}
\end{equation}
Since the discontinuity at $x=-1$ is entropic, and the data is a stationary solution on either side of the discontinuity, we can solve the PDE exactly by integrating the Rankine-Hugoniot ordinary differential equation
\begin{align*}
	\dot{y}(t)&=\sqrt{-y(t)}, \\
	y(0)&=-1,
\end{align*}
at least as long as $y<0$. However, the solution is easily seen to be $y(t)=-(t-2)^2/4$ for $t<2$, which means that at $t=2$, the solution has precise profile \eqref{shocking nonuniqueness}, and hence the Cauchy problem \eqref{claw} with flux $xu^2$ and $\L{\infty}$ initial datum \eqref{bounded nonuniqueness} also has infinitely many entropy solutions. Figure~\ref{linfblowup} illustrates the finite-time blow-up.
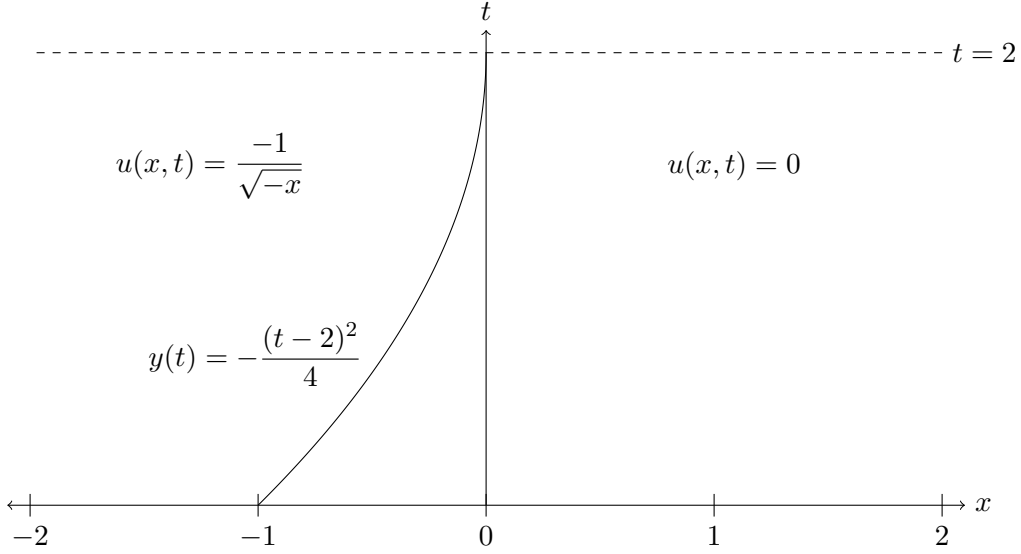
\begin{figure}[htbp]
	\centering
	\begin{tikzpicture}[scale=3]
		% Define domain
		\def\xmin{-2}
		\def\xmax{2}
		\def\tmin{0}
		\def\tmax{2}
		
		% Draw axes
		\draw[<->] (\xmin-0.1,0) -- (\xmax+0.1,0) node[right] {$x$};
		\draw[->] (0,0) -- (0,\tmax+0.1) node[above] {$t$};
		
		% Add axis labels
		\foreach \x in {-2,-1,0,1,2}
		\draw (\x,0.05) -- (\x,-0.05) node[below] {$\x$};
		
		% Draw domain boundary
		\draw[black, dashed] (\xmax,\tmax) -- (\xmin,\tmax);
		
		% Label t = 2
		\node[right] at (\xmax, \tmax) {$t=2$};
		
		% Label y(t)
		\node[left] at (-1/2,2/3) {$y(t)=-\dfrac{(t-2)^2}{4}$};
		
		% Label values
		\node[left] at (-0.75,1.5) {$u(x,t)=\dfrac{-1}{\sqrt{-x}}$};
		\node[right] at (0.75,1.5) {$u(x,t)=0$};
		
		% Blowup characteristic
		\draw[black] plot[domain=0:2, samples=200, variable=\t] ({-pow(\t-2,2)/4}, \t);
	\end{tikzpicture}
	\caption{The entropy solution to the Cauchy problem for \eqref{claw} with flux $f(x,u)=xu^2$ and initial data given by \eqref{bounded nonuniqueness} evolves into the profile \eqref{shocking nonuniqueness} at $t=2$, beyond which there are infinitely many entropic continuations.}\label{linfblowup}
\end{figure}
Let $u(x,t)$ be the simple entropy solution obtained by taking $y(t)=0$ for all $t\geq2$, and let $v(x,t)$ denote the entropy solution
\[
v(x,t)=
\begin{cases}
    -1/\sqrt{-x},\quad&\text{if }x<0, \\
    1/\sqrt{x},\quad&\text{if }0\leq x<t^2/4, \\
    0,\quad&\text{if }t^2/4\leq x,
\end{cases}
\]
as defined in \eqref{sqrt RH curve}. With this convention, one family of infinite solutions parametrised by $\tau\geq2$ is given by
\begin{equation}\label{fam}
    u_{\tau}(x,t)=
    \begin{cases}
        u(x,t),\quad&\text{if }t\leq\tau, \\
        v(x,t-\tau),\quad&\text{if }t>\tau,
    \end{cases}
\end{equation}
so that formally, $u_{\infty}=u$ with $u_\tau\to u$ in $\Lloc{1}$ as $\tau\to\infty$. We remark also that this family does not exhaust all possible entropy solutions of the Cauchy problem. For instance, the solution described in \eqref{fam} with $v$ defined instead as
\[
v(x,t)=
\begin{cases}
    -1/\sqrt{-x},\quad&\text{if }x<0, \\
    1/\sqrt{2x},\quad&\text{if }0\leq x<t^2/8, \\
    0,\quad&\text{if }t^2/8\leq x,
\end{cases}
\]
is also an entropy solution of \eqref{claw} and generates yet another one-parameter family of such solutions. Note also that there is nothing special about the quadratic exponent here, and we could have used any flux of the form $xu^s$ with $s>1$ to carry out the earlier computations and generate infinitely many entropy solutions. This is essentially due to the non-uniqueness of Cauchy problems for ODEs with merely Hölder continuous right-hand sides. For a flux of the form $xu^s$, the initial datum/stationary solution \eqref{shocking nonuniqueness} is simply
\[
u_0(x)=
\begin{cases}
	-1/\sqrt[s]{-x}&\text{ if }x<0, \\
	0&\text{ if }x>0,
\end{cases}
\]
and the non-stationary solution analogous to \eqref{sqrt RH curve} is now
\[
u(x,t)=
\begin{cases}
	-1/\sqrt[s]{-x}&\text{ if }x<0, \\
	1/\sqrt[s]{x}&\text{ if }0<x<y(t), \\
	0&\text{ if }x>y(t),
\end{cases}
\]
where
\[
y(t)=\left(\dfrac{s-1}{s}t\right)^{\frac{s}{s-1}}.
\]
Hence, even mild super-linearity can induce non-uniqueness. Finally, we remark that this phenomenon can also be extended to several space dimensions by a simple extension.
\begin{corollary}
	Consider the multidimensional scalar conservation law
	\[
	u_t+(x_1u^2)_{x_1}+\sum_{i=2}^{n}f_i(u)_{x_i}=0\text{ in }\R^n\times[0,\infty),
	\]
	where $n\geq2$ and $f_i$ are arbitrary $C^2$ functions. For initial data of the form
	\[
	u_0(x_1,\ldots,x_n)=
	\begin{cases}
		-x_1^{-1/2}&\text{ if }x_1\geq1, \\
		0&\text{ otherwise},
	\end{cases}
	\]
	there exist infinitely many solutions satisfying Kružkov's entropy inequalities.
\end{corollary}
Since the extension in higher dimensions is trivial, the corollary follows directly from the preceding construction.

\subsection{Non-existence of weak solutions}
Consider again the Cauchy problem for the following conservation law:
\begin{equation}\label{noweak}
	\begin{split}
		u_t+(xu^2)_x&=0, \\
		u(x,0)&=u_0(x).
	\end{split}
\end{equation}
In this section, we will once again consider initial data of the form
\begin{equation}\label{bad data}
	u_0(x)=
	\begin{cases}
		-1/\sqrt{x}&\text{ if }x>0, \\
		0&\text{ if }x<0.
	\end{cases}
\end{equation}
We claim that there are no weak solutions of \eqref{noweak} in the class of locally integrable functions $u$ with $\supnorm{f\circ u}=1$, where $f(x,u)=xu^2$. Note that $\supnorm{f\circ u_0}=1$, and entropy solutions of \eqref{claw} generally satisfy a flux-maximum principle. The method of proof will be to assume that a weak solution exists, and derive a contradiction. In order to do this, we first need some lemmas on the necessary conditions that any weak solution, if it exists, must satisfy.
\begin{lemma}[Trace at the interface]\label{strongtrace}
	Let $u$ be a weak solution of \eqref{noweak}. Then, $f\circ u(x,t)$ has zero trace at the interface $x=0$, i.e. for all $T>0$:
	\[
	\lim_{x\to0\pm} \int_0^T \abs{f\circ u(x, t)} dt = 0.
	\]
\end{lemma}
\begin{proof}
	Consider the test function $\varphi_{\eps}(x,t)=\eps\eta_{\eps}(x)\psi(t)$, where $\psi\in\Cc{\infty}([0,\infty))$ and $\eta_{\eps}$ is from a family of approximate identities, with $\eta$ supported in $[-1,1]$ and $\eta^{\prime}(x)\operatorname{sgn}(x)\geq0$. If $u$ is a weak solution, we can use $\varphi_{\eps}$ as a test function and obtain, for all $\eps>0$:
	\[
	\int_{0}^{\infty}\psi^{\prime}(t)\int_{-\eps}^{\eps}\eps\eta_{\eps}(x)u(x,t)dxdt=-\int_{0}^{\infty}\psi(t)\int_{-\eps}^{\eps}\eps\eta_{\eps}^{\prime}(x)f\circ u(x,t)dxdt.
	\]
	Since $f\circ u$ is non-negative (respectively, non-positive) if $x>0$ (respectively, $x<0$), and the term on the left-hand side vanishes as $\eps\to0$ (since the integrand is absolutely integrable and the domain shrinks in measure to zero), the term on the right-hand side must vanish too. Since $\psi$ was arbitrary, this concludes the proof of the strong trace property.
\end{proof}
The following corollary now follows immediately from Lemma~\ref{strongtrace}.
\begin{corollary}\label{restriction}
	Let $u$ be a weak solution of \eqref{noweak} with initial data $u_0$ such that $u_0(x)=0$ for $x<0$. Then, the function $\tilde{u}$ defined as
	\[
	\tilde{u}(x,t)=
	\begin{cases}
		u(x,t),&\text{ if }x\geq0, \\
		0,&\text{ if }x<0,
	\end{cases}
	\]
	is also a weak solution of the same Cauchy problem.
\end{corollary}
Hence, we will assume from now on that, without loss of generality, the weak solution $u$ is supported in the domain $\{x\geq0\}$. If $u$ agrees with the entropy solution outside a compact domain, we have the following result.
\begin{lemma}[Integral evolution]\label{int evo}
	Let $u$ be a weak solution of \eqref{noweak} with initial data \eqref{bad data} such that $u(x,t)=-x^{-1/2}$ for all $x>K,t\leq T$ for some positive constants $K,T$. Then, for almost all $0\leq t\leq T$:
	\[
	\int_{0}^{\infty}u(x,t)+\dfrac{1}{\sqrt{x}}dx=-t.
	\]
\end{lemma}
\begin{proof}
	Since the equality is trivially true for $t=0$, we will only consider the case of $t>0$. Furthermore, by Corollary~\ref{restriction}, we may assume without loss of generality that $u-u_0$ has compact support in $\mathbb{R}\times[0,T]$. Consider the domain $[0,K+1]\times[0,t]$, where $t<T$. By approximating the characteristic function of this domain, and assuming that $t$ is a Lebesgue point of $u:[0,T]\to\Lloc{1}(\R)$, we have that
	\[
	\int_{0}^{t}f\circ u(K+1,s)-f\circ u(0+,s)ds+\int_{0}^{K+1}u(x,t)-\dfrac{-1}{\sqrt{x}}dx=0.
	\]
	From the strong trace property of Lemma~\ref{strongtrace} and the fact that $u-u_0=0$ for $x>K$, the lemma follows, since $f\circ u_0(K+1)=1$.
\end{proof}
Next, we localise the weak solution in the following lemmas, showing that the far-field condition can be assumed without much loss of generality.
\begin{lemma}\label{localise}
	Let $\Omega_T=\R\times[0,T]$, and let $u$ be a weak solution of \eqref{noweak} with the initial data \eqref{bad data} such that $\supnorm{f\circ u}<\infty$, and $u$ has well-defined spatial traces in the domain $\{x>0\}$. Let $\supnorm{f\circ u}=\overline{C}$. Then, there exists $K>0$ and a (possibly different) weak solution $\tilde{u}$ such that $\tilde{u}(x,t)=-x^{-1/2}$ for all $x>K$.
\end{lemma}
\begin{proof}
	Without loss of generality, consider the left-continuous representative of $u$. By hypothesis, $u$ admits spatial traces in the domain $\{x>0\}$. Hence, the left-continuous representative of $u$ is well-defined here, and we can work with it throughout. Define the function $Q$ as follows:
	\[
	Q(x,t)=
	\begin{cases}
		\dfrac{f\circ u(x,t)-1}{u(x,t)+x^{-1/2}},&\text{ if }u(x,t)\neq-x^{-1/2}, \\[2ex]
		-2x^{1/2},&\text{ otherwise}.
	\end{cases}
	\]
	Then, if $\gamma(t)$ is the solution of the initial value problem
	\[
	\begin{split}
		\dot{\gamma}(t)&=Q(\gamma(t),t), \\
		\gamma(0)&=x_0>0,
	\end{split}
	\]
	we can define a new solution $\tilde{u}$ that satisfies the far-field criterion of Lemma~\ref{int evo} on any time interval $[0,T]$ where $\gamma(t)>0$ for all $T>0$ as follows:
	\[
	\tilde{u}(x,t)=
	\begin{cases}
		u(x,t)&\text{ if }x\leq\gamma(t), \\
		-x^{-1/2}&\text{ if }x>\gamma(t).
	\end{cases}
	\]
	Since the flux $f(x,u)=xu^2$ satisfies \eqref{LL}, the curve $\gamma$ is uniformly Lipschitz. Although $Q$ is discontinuous in $x$, it is uniformly bounded by \eqref{LL}, and along $\gamma$ it coincides at almost every time with either the genuine characteristic speed $f_u(\gamma(t),u(\gamma(t),t))$ or the Rankine-Hugoniot speed determined by the one-sided traces of $u$ at $\gamma(t)$, and depending on if they match or are unequal. The initial value problem is therefore to be understood in the Carathéodory sense, and $\gamma$ is either a Rankine-Hugoniot curve or a genuine characteristic; it is thus trivial to see that $\tilde{u}$ is a weak solution. 
	
	We claim that $\gamma(t)>0$ for all $t\in[0,T]$. Let $\gamma(t)>0$ for $t<t^*$ and $\gamma(t^*)=0$ for some $t^*\in[0,T]$. Then $\tilde{u}$ is a weak solution on $\Omega_{t^*}$. However, $\tilde{u}$ must violate the integral evolution identity of Lemma~\ref{int evo}, which is a contradiction. Hence, $\gamma(t)>0$ for all $t\in[0,T]$, and we are done.
\end{proof}
With this lemma, we can restrict ourselves to the case of weak solutions satisfying the far-field criterion of Lemma~\ref{int evo}. Hence, we have the following theorem on non-existence of weak solutions.
\begin{theorem}
	Let $u\in\Lloc{1}(\Omega_T)$ such that $u(x,0)=u_0(x)$, where $u_0$ is defined as in \eqref{bad data} and $t=0$ is a Lebesgue point of $u:[0,T]\to\Lloc{1}(\R)$. If $\supnorm{f\circ u}\leq1$, then $u$ \emph{cannot} be a weak solution of the Cauchy problem \eqref{noweak}.
\end{theorem}
\begin{proof}
	Note that $\supnorm{f\circ u_0}=1$ and therefore the inequality condition is really an equality.
	
	Suppose, if possible, that $u$ is a weak solution of \eqref{noweak}. By the Lemmas~\ref{strongtrace} and \ref{localise}, we may assume that $u(x,t)-u_0(x)$ is zero for $x\in[0,K]^c$ for some $K$. That is, we may assume without loss of generality that $u$ satisfies the far-field criterion of Lemma~\ref{int evo}. However, if $\supnorm{f\circ u}\leq1$, then for $x>0$ we have that
	\[
	u(x,t)\geq-\dfrac{1}{\sqrt{x}}\implies u(x,t)+\dfrac{1}{\sqrt{x}}\geq0,
	\]
	but this is precisely the integrand in Lemma~\ref{int evo}. Since the integral, however, must be negative at all positive times, this is a contradiction. Hence, $u$ cannot be a weak solution of \eqref{noweak} with initial data $u_0$ if $\supnorm{f\circ u}\leq1$.
\end{proof}

\subsection{A sufficient condition for blow-up of entropy solutions}\label{blowup}
Consider a general flux of the form
\begin{equation}\label{general flipping flux}
	f(x,u)=g(x)h(u),
\end{equation}
where $h,g$ satisfy the structural assumptions outlined in section~\ref{ass}. Note that the flux $f(x,u)=xu^2$ considered earlier is a special case of such a flux. The characteristic system \eqref{char ode} reduces to
\begin{align}
	\dot{q}(t)&=g(q)h^{\prime}(p),\label{genflux char q} \\
	\dot{p}(t)&=-g^{\prime}(q)h(p).\label{genflux char p}
\end{align}
For fluxes of this form, we have the following result.
\begin{theorem}\label{vander}
	Consider the Cauchy problem \eqref{claw} for a flux with multiplicative heterogeneity of the form \eqref{general flipping flux}. If, for some zero $x_0$ of $g$ we have that either $g^{\prime}(x_0+)\neq0$ or $g^{\prime}(x_0-)\neq0$, then there exist compactly supported initial data $u_0$ such that the $\L{\infty}$ norm of the entropy solution $u$ blows up in finite time. In particular, we can also choose  $u_0\in \Cc{1}(\R)$.
\end{theorem}
\begin{proof}
	Without loss of generality, assume that $x_0=0$ and $g^{\prime}(0+)>0$. Since $g$ is $\Ck{2}$ up to the boundary, there exist $K,\beta>0$ such that for all $x\in(0,K]:g(x)>0,g^{\prime}(x)\geq\beta$. Note that $f$ is now strictly convex if $K>x>0$. Let $m>0$, and consider the characteristic equations \eqref{genflux char q},\eqref{genflux char p} with initial conditions
	\[
	q(0)=y\in(0,K), \quad p(0)=-m
	\]
	We claim that $p$ must blow up in finite time. The forward characteristics $q$ emanating from positive $y$ travel at negative speed. Furthermore, since $g$ is Lipschitz with constant $L$ (say), and also $g^\prime\geq\beta$, from \eqref{genflux char p} we have that
	\[
	-Lh(p(t))\leq\dot{p}(t)\leq-\beta h(p(t)).
	\] 
	In particular, $p$ cannot blow up instantaneously if $p(0)=-m$. However, by uniform convexity of $h$ we have that
	\[
	\dot{p}(t)\leq\dfrac{\alpha\beta}{2}p(t)^2,
	\]
	and hence must blow up in finite time. By the conservation of $f$ along characteristic trajectories, this can only happen when $q(t)=0$. Hence, if we can construct initial data such that distinct forward characteristics do not intersect except at $x=0$, we are done. Fix $y_0\in(0,K)$ and $m_0>0$, and let $C_0=g(y_0)h(-m_0)$. Since $f(q(t),p(t))=C_0$ along the trajectory and $p$ does not change sign, we have that
	\[
	p(t)=h_{-}^{-1}(C_0g(q(t))^{-1})
	\]
	for $q(t)>0$, where $h_{-}^{-1}$ is the negative branch of the inverse of $h$. Hence, we obtain an autonomous equation for $q$ of the form
	\[
	\dot{q}(t)=g(q(t))h^{\prime}\circ h_{-}^{-1}(C_0g(q(t))^{-1}),
	\]
	and since $q(\tau)=0$ for some $\tau>0$, namely the time taken for the solution to blow up, we can derive that
	\[
	-\tau=\int_{0}^{y_0}\dfrac{dq}{g(q)h^{\prime}\circ h_{-}^{-1}(C_0g(q)^{-1})}.
	\]
	Since $m_0,y_0$ were arbitrary, we can think of the blow-up time $\tau$ as a function of the arguments $m,y$ or equivalently $C,y$, i.e.
	\[
	\tau(m,y)=\int_{0}^{y}\dfrac{-dq}{g(q)h^{\prime}\circ h_{-}^{-1}(g(y)h(-m)g(q)^{-1})},\quad\tilde{\tau}(C,y)=\int_{0}^{y}\dfrac{-dq}{g(q)h^{\prime}\circ h_{-}^{-1}(Cg(q)^{-1})}.
	\]
	Note that $C$ is a monotone strictly increasing function of $m$. If we differentiate $\tilde{\tau}$ with respect to $C$, then it is clear from the strict convexity of $h$ that $\partial_C\tilde{\tau}(C,y)<0$; as $C\to0$ we have $\tilde{\tau}(C,y)\to+\infty$, and as $C\to+\infty$, we have $\tilde{\tau}(C,y)\to0$. Now, for each $y\in(0,y_0)$, there exists a corresponding $m(y)$ such that $\tau(m(y),y)=\tilde{\tau}(C_0,y_0)$ where $C=g(y)h(-m)$, and by the implicit function theorem it is $\Ck{1}$. Furthermore, the comparison principle argument from before lets us know that $m(y)$ is uniformly bounded in $[0,y_0]$, where $m(0)$ is defined as the limit $m(y)$ as $y\to0$. Now, define $u_0$ as follows:
	\[
	u_0(x)=
	\begin{cases}
		-m(x),&\text{ if }x\in[0,y_0], \\
		-\tilde{m}(x)&\text{ otherwise},
	\end{cases}
	\]
	where $\tilde{m}(x)\geq0$ is a $C_c^1$ continuation of $m$ supported in $[-K,K]$ such that $\tau(x,m(x))$ is strictly increasing in $x$ for $x\in(y_0,K)$. One simple way of accomplishing this is to ensure that $C(x)=g(x)h(-\tilde{m}(x))$ is strictly decreasing for $x>y_0$ and vanishes at $x=K$, which implies that $\tilde{m}$ decreases to zero since $g(x)$ is increasing. The behaviour of $\tilde{m}$ for $x<0$ can without loss of generality be chosen to be any monotone increasing smooth function vanishing for $x\leq-K$, since characteristics do not cross the $x=0$ boundary. Thus, it is enough to show that blow-up happens in the domain $x\geq0$; if the characteristics emanating from $x\in(0,y_0)$ do not interact in the domain $\{x>0\}$, we have $\L{\infty}$ blow-up at $(x,t)=(0,\tau)$, where we abuse notation slightly and denote $\tau:=\tau(m_0,y_0)$, and this is what we show next.
	
	Recall that the characteristics $q$ satisfy the autonomous equation
	\[
	\dot{q}(t)=V(q(t),C),
	\]
	for some $C>0$ where
	\[
	V(q,C)=g(q)h^{\prime}\circ h_{-}^{-1}(Cg(q)^{-1}).
	\]
	In particular, we have $C=g(q(0))h(u_0(q(0)))$. Now suppose $q_1,q_2$ are two distinct characteristics with $0<q_1(0)<q_2(0)<y_0$. Then, denoting $y_i:=q_i(0)$, we have that $p_i(0)=-m(y_i)$ for the corresponding value functions along the trajectories. Let $C(y)=g(y)h(-m(y))$, then we have that $C(y_1)<C(y_2)$; for notational convenience let $C_i:=C(y_i)$. Now, for $i=1,2$, we have that, by construction
	\[
	-\tau=\int_{0}^{y_i}\dfrac{dq}{V(q,C_i)}.
	\]
	Hence, subtracting one from the other yields
	\begin{align}\label{incompat}
		\int_{0}^{y_1}\left(\dfrac{1}{V(q,C_1)}-\dfrac{1}{V(q,C_2)}\right)dq=\int_{y_1}^{y_2}\dfrac{dq}{V(q,C_2)}.
	\end{align}
	Suppose, if possible, that $0<q_1(r)=q_2(r)=\overline{x}$ for some $r<\tau$. Since the characteristic speeds are negative, $\overline{x}<y_1$. By the same integral representation as before, we have that for $i=1,2$:
	\[
	-r=\int_{\overline{x}}^{y_i}\dfrac{dq}{V(q,C(y_i))}.
	\]
	Hence, subtracting yields
	\begin{align}\label{2incompat}
		\int_{\overline{x}}^{y_1}\left(\dfrac{1}{V(q,C_1)}-\dfrac{1}{V(q,C_2)}\right)dq=\int_{y_1}^{y_2}\dfrac{dq}{V(q,C_2)}.
	\end{align}
	However, since the right-hand sides of \eqref{incompat} and \eqref{2incompat} are the same, we conclude that
	\[
	\int_{\overline{x}}^{y_1}\left(\dfrac{1}{V(q,C_1)}-\dfrac{1}{V(q,C_2)}\right)dq=\int_{0}^{y_1}\left(\dfrac{1}{V(q,C_1)}-\dfrac{1}{V(q,C_2)}\right)dq,
	\]
	which implies that
	\begin{align}\label{contradiction}
		\int_0^{\overline{x}}\left(\dfrac{1}{V(q,C_1)}-\dfrac{1}{V(q,C_2)}\right)dq=0.
	\end{align}
	Recall however that $C_1<C_2$, and that $V$ is strictly monotone with respect to $C$. Hence, the integrand is strictly negative, which contradicts \eqref{contradiction} if $\overline{x}>0$. This completes the proof: distinct characteristics can (and do) only meet at $x=0$.
\end{proof}

\subsection{A sufficient condition for non-uniqueness of entropy solutions}\label{sec nonuni}
$\L{\infty}$ blow-up goes along with the non-uniqueness, and hence this is also a sufficient condition for the existence of initial data with infinitely many Kružkov entropy solutions.
\begin{theorem}\label{generic nonuni}
	Consider the Cauchy problem \eqref{claw} for a flux with multiplicative heterogeneity of the form \eqref{general flipping flux}. If, for some zero $x_0$ of $g$ the one-sided derivatives $g^{\prime}(x_0-)$ and $g^{\prime}(x_0+)$ exist and are non-zero with the same sign, which is to say that $g$ changes sign at $x_0$, then there exists initial data $u_0\in\L{\infty}(\R)$ such that the Cauchy problem has infinitely many Kružkov entropy solutions.
\end{theorem}
\begin{proof}
	Without loss of generality, suppose $x_0=0$ and $g^{\prime}(0-),g^{\prime}(0+)>0$, so that $g<0$ on $(-\delta,0)$ and $g>0$ on $(0,\delta)$ for some $\delta>0$. We will imitate the construction for $f(x,u)=xu^2$; since the construction is localised in space and $g$ is $\Ck{2}$ up to the boundary on each side of the origin by \eqref{R}, we may assume that $g^{\prime}\geq\beta>0$ on $[-K,0)\cup(0,K]$ for some positive constant $\beta$. Let $u_+$ denote the continuous stationary solution of the Cauchy problem of such that $u_r(x)\geq0$ with $u_r(x)=0$ for $x<0$ and $f(x,u_r(x))=1$ for $x>0$. Similarly, let $u_l(x)$ denote the continuous stationary negative solution supported on $(-\infty,0]$ with $f(x,u_l(x))=-1$ for $x<0$. First, we will show that function $u(x,t)$ is an entropy solution of \eqref{claw}, where $u(x,t)=u_l(x)+u_r(x)$, or in other words
	\[
	u(x,t)=
	\begin{cases}
		u_l(x)&\text{ if }x<0, \\
		u_r(x)&\text{ if }x>0.
	\end{cases}
	\]
	It is clear that $u$ is a classical solution of \eqref{claw} away from the interface $\{x=0\}$, and by \eqref{V}-\eqref{CG} $u$ and $f\circ u$ are also locally integrable. Hence, for any non-negative $\varphi\in\Cc{1}(\R\times(0,\infty))$ and $k\in\R$, the same computation as was done in Lemma~\ref{stat sol} shows that $u$ as defined is an entropy solution of \eqref{claw}. For completeness, we replicate the computation here.
	\begin{align*}
		E:=& \iint_{\R^2_+}\abs{u(x)-k}\varphi_t+\sgn(u(x)-k)\left\{\left[f(x,u(x))-f(x,k)\right]\varphi_x-f_x(x,k)\varphi\right\}dxdt \\
		=&\int_0^{\infty}\int_{-\eps}^{\eps}\abs{u(x)-k}\varphi_t+\sgn(u-k)\left\{\left[f(x,u)-f(x,k)\right]\varphi_x-f_x(x,k)\varphi\right\}dxdt \\
		&-\int_{0}^{\infty}\left[f(\eps,u)-f(\eps,k)\right]\varphi(\eps,t)dt-\int_{0}^{\infty}\left[f(-\eps,u)-f(-\eps,k)\right]\varphi(-\eps,t)dt.
	\end{align*}
	As $\eps\to0$, by the construction of $u$, we obtain:
	\[
	E=\int_{0}^{\infty}(-f(1,u(1))-f(-1,u(-1)))\varphi(0,t)dt=0.
	\]
	Now, consider the initial data $u_0(x)=u_l(x)$. By the blow-up criterion of Theorem~\ref{vander}, we can obtain this profile at a positive time from $\L{\infty}$ initial data. Hence, it is sufficient to demonstrate that the Cauchy problem with initial data $u_l(x)$ has infinitely many entropy solutions. Note that $u_l$ itself is already a Kružkov entropy solution. Thus, if $\gamma(t)$ denotes the Rankine-Hugoniot curve connecting the stationary solution $u_r$ with zero to the right with $\gamma(t)>0$ for $t>0$ and $\gamma(t)=0$ for $t\leq0$, we can define an infinite family of entropy solutions to the Cauchy problem $u_{\lambda}(x,t)$  parametrised by $\lambda\in[0,\infty)$ as follows:
	\[
	u_{\lambda}(x,t)=
	\begin{cases}
		u_l(x)&\text{ if }x<0, \\
		u_r(x)&\text{ if }0\leq x<\gamma(t-\lambda), \\
		0&\text{ if }\gamma(t-\lambda)\leq x.
	\end{cases}
	\]
\end{proof}
Hence, the sufficient condition for blow-up as established in Theorem~\ref{vander} is also a sufficient condition for the non-uniqueness of entropy solutions. Kružkov's `doubling of variables' argument \cite[Theorem 1]{kruzkov} works only in the $\L{\infty}$ setting.

\section{Front tracking approximations}\label{ft approx}
Despite the pathologies uncovered in the preceding sections, we can select a unique entropy solution of \eqref{claw} as the limit of front-tracking approximations, suitably defined. Recall that $G=\{x:g(x)=0\}$ is a closed set, by the continuity of $g$ and therefore the complement of $G$ is a countable union of disjoint open intervals, say $G^c=\cup(a_i,b_i)$ where the possibility that $a_i=-\infty,b_i=+\infty$ are not excluded. By the assumption \eqref{B}, we also know that these boundary points cannot accumulate.

$\L{\infty}$ bounds fail to hold without coercivity, so numerical schemes such as \cite{sylla2024convergencefinitevolumescheme} break down (see section~\ref{blowup}). Hence, we employ a novel front tracking scheme to demonstrate the well-posedness of \eqref{claw}. Consider a flux of the form \eqref{mult_form} such that $g$ is supported on an interval $(a,b)$, possibly with $a=-\infty$ or $b=+\infty$. Without loss of generality we assume that $g>0$ here, the case $g<0$ can be analysed in similar fashion.

Since the flux is uniformly zero outside the domain $(a,b)$, the solution is stationary regardless of initial datum, and thus we can restrict ourselves to the case of $u_0$ supported in $(a,b)$ without any loss of generality. We can demonstrate the existence of solutions by means of a heterogeneous front-tracking scheme as developed in \cite{hetft}, since $f$ is strictly convex on the open interval and uniformly convex on compact sets within it. The steps are detailed as follows:
\begin{enumerate}
	\item Approximate the initial data $u_0$ by piecewise smooth functions $u_0^{\delta}$ such that ${\cF(x,u_0^{\delta})}=\sgn(u_0^{\delta}(x))f(x,u_0^{\delta}(x))$, which we denote by ${\cF^{\delta}}$, is piecewise constant and converges to $\cF(x,u_0(x))$ in $\Lloc{1}(\R)$ as $\delta\to0$. That is, we discretise the (signed) flux-value of the initial data rather than the data itself directly.
	\item For each $\delta$, solve the Riemann problems at the initial time either as Rankine-Hugoniot shocks or $\delta$-fans \cite[Definition 1]{hetft}.
	\item Whenever interactions occur, the forward solution is a Rankine-Hugoniot shock in all cases \cite[Theorem 2.4]{hetft}.
\end{enumerate}
The only additional difficulty in this case is handling the interfaces where $g$ vanishes. Since we will work locally for the front tracking, it is enough to consider the case where $b=+\infty$ and, without loss of generality, $a=0$. Thus in this case $g$ is a positive function on the positive real numbers. Set 
\begin{equation}\label{interface riem}
	u_0(x)=
	\begin{cases}
		0&\text{ if }x<0, \\
		u_r(x)&\text{ if }x\geq0,
	\end{cases}
\end{equation}
where $u_r$ is a differentiable function such that $x \mapsto f(x, u_r(x))$ is constant, say equal to $f_r$. For any positive number $f_r\in\R$ we can find exactly two such functions satisfying our requirement, one positive and the other negative, by the implicit function theorem. If $f_r=0$, then $u_r\equiv0$ is the only such function. These are precisely the stationary solutions corresponding to the flux \cite[Lemma 2.1]{hetft}. We have the following important lemma which assures us that $u_0$ is locally integrable.
\begin{lemma}\label{gen riem}
	When the flux satisfies \eqref{V} and \eqref{CG}, generalised Riemann data of the form \eqref{interface riem} are integrable near the interface. More generally, if $x \mapsto f(x, u_0(x))$ is bounded, then $u_0$ is locally integrable.
\end{lemma}
\begin{proof}
	Let $f_r>0$ and consider the negative stationary solution $u_r$ associated with $f_r$. Then, since $f(x,u_r(x))=f_r$ for all positive $x$, we have that
	\[
	h(u_r(x))=\left(\dfrac{f_r}{g(x)}\right).
	\]
	Now, if $p$ is the order of vanishing for $g$ in \eqref{V}, then near zero $g$ behaves like $x \mapsto x^{p}$, while by \eqref{CG}, we have that $\abs{y}\leq h((C^{-1}y)^{1/(p+\eps)})$ for $\abs{y}$ large enough. Hence, for some constants $\tilde{C}>0$ and $\theta\in(0,1)$, we have that
	\[
	u_r(x)\leq\dfrac{\tilde{C}}{x^{1-\theta}},
	\]
	and thus the singularity at $x=0$ is integrable. Since $u_0$ is uniformly bounded outside a neighbourhood of $0$ if $f$ is, the local integrability of more general functions supported on the positive real numbers trivially follows.
\end{proof}
In section~\ref{non int} we demonstrate that the growth exponent condition on $h$ is sharp with respect to the behaviour of $g$ near its zeroes.

We claim that if $u_0<0$, then the stationary solution $u(x,t)=u_0(x)$ is an entropy solution of the Cauchy problem \eqref{claw} associated with initial datum \eqref{interface riem}. It is trivial to see that, away from the interface, $u$ is a classical solution of \eqref{claw}. Let $k\in\R$ be given. Now, $u_0$ blows up near $x=0$; in particular, $u_0(x)\to-\infty$ as $x\to0$. Hence, there exists $N>0$ such that $u_0(x)<k$ for all $x\in(0,N)$. Taking $\eps\in(0,N)$ arbitrarily small, and noting that $u(x,t)$ is a classical solution of \eqref{claw} for $x>0$, we have that for any positive test function $\varphi$,
\begin{align*}
	L=&\iint_{\R^2_+}\abs{u-k}\varphi_t+\sgn(u-k)\left[f(x,u)-f(x,k)\right]\varphi_xdxdt \\
	=&\int_{0}^{\eps}\int_{0}^{\infty}\abs{u-k}\varphi_t+\sgn(u-k)\left[f_r-f(x,k)\right]\varphi_xdxdt \\
	&+\int_{\eps}^{\infty}\int_{0}^{\infty}\abs{u-k}\varphi_t+\sgn(u-k)\left[f(x,u)-f(x,k)\right]\varphi_xdxdt.
\end{align*}
Since $u(x,t)=u_0(x)$ is a classical solution of \eqref{claw} for $x\geq\eps$, we can apply integration by parts on the second integral. Now, $u_0(\eps)<k$, hence $\sgn(u-k)<0$ on this boundary, and thus
\begin{align*}
	L=&\int_{0}^{\eps}\int_{0}^{\infty}\abs{u-k}\varphi_t+\sgn(u-k)\left[f_r-f(x,k)\right]\varphi_xdxdt \\
	&+\int_{0}^{\infty}[f_r-f(\eps,k)]\varphi(\eps,t)dt \\
	&+\int_{\eps}^{\infty}\int_{0}^{\infty}\sgn(u-k)f_x(x,k)\varphi dxdt-\int_{\eps}^{\infty}\abs{u_0(x)-k}\varphi(x,0)dx.
\end{align*}
Now, as $\eps\to0$, the first integral vanishes by the local integrability of the integrand, and the second integral approaches
\[
f_r\int_{0}^{\infty}\varphi(0,t)dt,
\]
since $f(\eps,k)\to0$ as $\eps\to0$ for any fixed $k\in\R$. Since $f_r>0$, the test function $\varphi$ and real number $k$ were arbitrarily chosen and the integrand at the initial time is non-negative even for $x<0$, we can conclude that \eqref{entweak} holds; the stationary solution $u(x,t)=u_0(x)$ is an entropy solution of \eqref{claw} with initial datum \eqref{interface riem} when $u_0<0$, and the same computation shows that this is no longer true when $u_0>0$.

In order to avoid the trouble with positive stationary solutions, we ensure that our initial discretisation {$\cF^{\delta}$} is always $0$ in a small neighbourhood of the origin, which vanishes as the discretisation parameter $\delta\to0$. Then, fronts only interact with the interface to produce entropic interfaces. However, note that solutions satisfying the entropy condition \eqref{entweak} need not satisfy the weak formulation of \eqref{claw}; if they happen to be uniformly bounded in addition to satisfying the entropy condition, then we can say that they are weak solutions by considering $k=\pm\supnorm{u}$, but not otherwise.

\subsection{Flux with discrete zero set}
Since solutions can be analysed in each maximal open interval, it is sufficient to generalise the previous case where $G$ is a singleton, say $G=\{0\}$, by working within each maximal open domain where $g$ is non-negative. If $f$ satisfies the structural constraint with respect to all the zeroes in $G$, then the Cauchy problem \eqref{claw} has locally integrable solutions for flux-bounded initial data. More precisely, we have the following result:
\begin{theorem}
	Let $f$ be given by \eqref{mult_form} satisfy our assumptions and structural constraint. Fix $u_0\in \Lloc{1}(\R)$ such that the function $\theta(x)=\sgn(u_0(x))f(x,u_0(x))$ is locally of bounded variation. Then, the Cauchy problem \eqref{claw} has an entropy solution in the sense of \eqref{entweak} globally in time.
\end{theorem}

\subsection{Uniqueness and stability}
We adopt Kružkov's `doubling of variables' technique in order to prove the uniqueness and stability of solutions. The main difficulty arises in dealing with merely integrable rather than uniformly bounded functions. We say that $u$ satisfies the interface condition on $G\times(0,T)$ if for each maximal bounded interval $(a,b)\subseteq G^c$, we have one of the following: 
\begin{enumerate}
	\item If $g>0$ in $(a,b)$ then $u(a+,t)<+\infty,u(b-,t)>-\infty$ for $t\in(0,T)$.
	\item If $g<0$ in $(a,b)$ then $u(a+,t)>-\infty,u(b-,t)<+\infty$ for $t\in(0,T)$.
\end{enumerate}
Our front tracking algorithm seems to indicate that the domain of dependence for solutions is localised within each maximal interval where $g$ is non-zero. Hence, limits of front tracking approximations all satisfy this interface condition. Furthermore, non-uniqueness involves characteristics emanating from the degenerate interface at positive times. Thus, we expect that such a Lax-type condition on the solution should allow us to recover uniqueness. The following theorem confirms this intuition.

\begin{theorem}[Uniqueness and stability]\label{unistabi}
	Let $u,v$ be entropy solutions of \eqref{claw} in a domain containing $\Omega_T$ satisfying the interface condition on $G\times(0,T)$ with respective initial values $u_0,v_0$. Then, for any $a,b\in\R$ such that $a<b$ with $g(a)=g(b)=0$ and $g(x)\neq0$ for $x\in(a,b)$, we have that for all $T>0$:
	\begin{equation}\label{stability}
		\int_{a}^{b}\abs{u(x,T)-v(x,T)}dx\leq\int_{a}^{b}\abs{u_0(x)-v_0(x)}dx.
	\end{equation}
\end{theorem}
\begin{proof}
	Our method of proof is to use the following inequality for positive test functions $\varphi$ supported in the interior of $\Gamma_T:=[a,b]\times[0,T]$.
	\begin{equation}\label{double variable}
		-\iint_{\Gamma_T}\abs{u-v}\varphi_t+\sgn(u-v)\left[f(x,u)-f(x,v)\right]\varphi_xdxdt\leq0.
	\end{equation}
	Let $S$ denote the compact support of $\varphi$. Since $f\circ u,f\circ v$ are bounded, $u,v\in \L{\infty}(S)$ and thus \eqref{double variable} follows from the arguments of Kružkov \cite[Theorem 1]{kruzkov}, which only requires regularity of the flux and boundedness of the entropy solutions. Now we choose a special family of test functions $\varphi_{\eps}$ as follows: let $\eta_{\eps}$ denote the standard, compactly supported approximations of identity, and let $\delta_{\eps}$ denote the function
	\[
	\delta_{\eps}(x)=\int_{-\infty}^{x}\eta_{\eps}(z)dz,
	\]
	so that it is strictly increasing, $\delta_{\eps}(y)=0$ for $y\leq-\eps$, and $\delta_{\eps}(y)=1$ for $y\geq\eps$. Then, for $0<s<r<T$ and $\eps>0$ small enough, define
	\[
	\varphi_{\eps}(x,t)=\left[\delta_{\eps}(t-s)-\delta_{\eps}(t-r)\right]\left[\delta_{\eps}(x-(a+2\eps))-\delta_{\eps}(x-(b-2\eps))\right],
	\]
	i.e. $\varphi_{\eps}\in \Cc{\infty}$ approximates the characteristic function of the rectangular domain $(a,b)\times(s,r)$. We can write this as $\varphi_{\eps}(x,t)=\tau_{\eps}(t)\psi_{\eps}(x)$, and note that $\tau_{\eps},\psi_{\eps}$ approximate the characteristic functions of the intervals $[s,r]$ and $[a,b]$ respectively. Now, for each $\eps>0$, the entropy solutions $u,v$ are uniformly bounded on the support of $\varphi_{\eps}$. Since the integrand in \eqref{double variable} consists of uniformly bounded terms for any fixed $\eps>0$ small enough, we can rewrite the inequality as follows:
	\begin{equation}\label{apprx double}
		\begin{split}
			&\int_{r-\eps}^{r+\eps}\eta_{\eps}(t)\int_{a}^{b}\abs{u-v}\psi_{\eps}(x)dxdt-\int_{s-\eps}^{s+\eps}\int_{a}^{b}\abs{u-v}\psi_{\eps}(x)dxdt \\
			\leq&\int_{a+\eps}^{a+3\eps}\eta_{\eps}(x)\int_{s-\eps}^{r+\eps}\sgn(u-v)[f(x,u)-f(x,v)]\tau_{\eps}(t)dtdx \\
			&-\int_{b-3\eps}^{b-\eps}\eta_{\eps}(x)\int_{s-\eps}^{r+\eps}\sgn(u-v)[f(x,u)-f(x,v)]\tau_{\eps}(t)dtdx.
		\end{split}
	\end{equation}
	Now, if $u\neq v$, then
	\[
	\sgn(u-v)[f(x,u)-f(x,v)]=\abs{u-v}\dfrac{f(x,u)-f(x,v)}{u-v},
	\]
	so by the interface condition on $u,v$ and the mean value theorem, the right-hand side of \eqref{apprx double} must converge to a non-positive quantity as $\eps\to0$, since all the other terms in the integrand(s) are positive. Taking $\eps\to0$ on the left-hand side as well and letting $s\to0,r\to T$, we obtain \eqref{stability} as required, which completes the proof.
\end{proof}
Hence, although entropy solutions of \eqref{claw} may not be unique, the limits of front-tracking approximations are. We remark that Theorem~\ref{unistabi} also allows us to extend the well-posedness theory to any initial data $u_0$ such that $f\circ u_0$ is merely uniformly bounded.

For unbounded intervals, the interface condition is replaces with a domain-of-dependence condition. Recall that \eqref{LL} guarantees flux-bounded solutions to satisfy a finite speed condition on compact sets, or more generally wherever $g$ is uniformly bounded. Although uniform bounds on the speed of propagation are not necessary, some form of growth restrictions on the heterogeneous term are essential for well-posedness, as we demonstrate in section ~\ref{infspeed}.

Let $h_{\pm}$ denote the positive and negative branches of $h$. Consider the characteristic equations \eqref{char ode} with $f(q(0),p(0))=M$, and suppose without loss of generality that $M>0,p(0)>0$. Since $f$ is conserved along trajectories and $p$ does not change sign, we can rewrite the characteristic equation for $q$ as
\[
\dot{q}(t)=g(q(t))h^{\prime}\circ h_{+}(Mg(q(t))^{-1}),
\]
as long as $g(q(t))$ is positive. Hence, we can invert this equation to determine the time taken by a characteristic to reach one point from another. We use this to ascertain when a characteristic may reach $\pm\infty$ in finite time: define
\begin{equation}\label{growth cond}
	V_M^+(x)=\int_{x}^{+\infty}\dfrac{dy}{g(y)h^{\prime}\circ h_{-}^{-1}(Mg(y)^{-1})}, \qquad V_M^-(x)=\int_{-\infty}^{x}\dfrac{dy}{g(y)h^{\prime}\circ h_{+}^{-1}(Mg(y)^{-1})}
\end{equation}
If $G^c$ contains a maximal unbounded interval $(a,b)$, i.e. either $a=-\infty$ or $b=+\infty$, the interface conditions are replaced by the following growth conditions on the flux.
\begin{enumerate}
	\item If $g>0$ in $(a,b)$ and $a=-\infty$, then for all $x\in(a,b),M>0:V_M^-(x)=+\infty$.
	\item If $g>0$ in $(a,b)$ and $b=+\infty$, then for all $x\in(a,b),M>0:V_M^+(x)=+\infty$.
	\item If $g<0$ in $(a,b)$ and $a=-\infty$, then for all $x\in(a,b),M<0:V_M^-(x)=-\infty$.
	\item If $g<0$ in $(a,b)$ and $b=+\infty$, then for all $x\in(a,b),M<0:V_M^+(x)=+\infty$.
\end{enumerate}
We remark that in each case, it is sufficient for the condition to be satisfied for all admissible $M$ (either all positive or all negative $M$ depending on the sign of $g$) and some $x\in(a,b)$. The substance of these non-integrability assumptions is that genuine characteristics may not reach $\pm\infty$ in finite time, thus they can be interpreted as ``infinity interface'' conditions. We state the following stability theorem assuming without loss of generality that $(a,b)=(-\infty,0)$ and $g>0$ on $(a,b)$. Analogous results hold in the other cases as well.
\begin{theorem}\label{infty stability}
	Let $u,v$ be entropy solutions of \eqref{claw} in a domain containing $(-\infty,0)\times[0,T]$ and respective initial values $u_0,v_0$, with $(-\infty,0)\subseteq G^c$ being a maximal open interval, i.e. $g(0)=0$ and $g(x)>0$ for all $x<0$, and suppose $u,v$ satisfy the interface condition at $0$. Further suppose that $\supnorm{f\circ u},\supnorm{f\circ v}<M$ (note that the inequality is strict). If $V_M^{-}(-1)=+\infty$, then there exists a monotone non-decreasing function $R_{-}:(-\infty,0)\to(-\infty,0)$ depending on $M$ and $f$ such that for all $x\in(-\infty,0)$:
	\[
	\int_x^{0}\abs{u(x,T)-v(x,T)}dx\leq\int_{R_{-}(x)}^{0}\abs{u_0(x)-v_0(x)}dx.
	\]
	Thus, in particular,
	\[
	\int_{-\infty}^{0}\abs{u(x,T)-v(x,T)}dx\leq\int_{-\infty}^{0}\abs{u_0(x)-v_0(x)}dx.
	\]
\end{theorem}
\begin{proof}
	In the proof of Theorem~\ref{stability}, we showed in \eqref{double variable} that for test functions compactly supported in $G^c\times(0,T)$:
	\[
	-\iint\abs{u-v}\varphi_t+\sgn(u-v)\left[f(x,u)-f(x,v)\right]\varphi_xdxdt\leq0.
	\]
	Since $(-\infty,0)\subseteq G^c$, we can use \eqref{double variable} here too, by choosing test functions carefully. Given $x<0$, let $q(t)$ denote the backward characteristic corresponding to the level $M$ with positive speed such that $q(T)=x$, i.e. $q(t)$ solving the terminal value problem
	\begin{align*}
		\dot{q}(t)&=g(q(t))h^{\prime}\circ h_{+}(Mg(q(t))^{-1}), \quad t\in[0,T], \\
		q(T)&=x.
	\end{align*}
	For each $x\in(-\infty,0)$, there exists such a curve, and we define $R(x)=q(0)$ for the corresponding curves. Crucially, since $V_M^{-}(x)=+\infty$ for all $x\in(-\infty,0)$, the function $R(x)$ is well defined and monotone non-decreasing. Now for $t\in[0,T]$, let $S(t)=\{x:q(t)<x<0\}$ and let $S$ be the union of all $S(t),t\in[0,T]$. As in the proof of Theorem~\ref{unistabi}, we approximate the characteristic function of the domain $S$ by test functions. Let $\eta_{\eps}$ denote the standard, compactly supported approximations of identity, and let $\delta_{\eps}$ denote the function
	\[
	\delta_{\eps}(x)=\int_{-\infty}^{x}\eta_{\eps}(z)dz,
	\]
	so that it is strictly increasing, $\delta_{\eps}(y)=0$ for $y\leq-\eps$, and $\delta_{\eps}(y)=1$ for $y\geq\eps$. Then, for $0<s<r<T$ and $\eps>0$ small enough, define
	\[
	\varphi_{\eps}(x,t)=[\delta_{\eps}(t-s)-\delta_{\eps}(t-r)][\delta_{\eps}(x-q(t))-\delta_{\eps}(x+2\eps)].
	\]
	We can write this as $\varphi_{\eps}(x,t)=\tau_{\eps}(t)\psi_{\eps}(x,t)$, so that
	\begin{align*}
		\partial_t\varphi_{\eps}(x,t)&=[\eta_{\eps}(t-s)-\eta_{\eps}(t-r)]\psi_{\eps}(x,t)+\tau_{\eps}(t)\eta_{\eps}(x-q(t))(-\dot{q}(t)), \\
		\partial_x\varphi_{\eps}(x,t)&=\tau_{\eps}(t)[\eta_{\eps}(x-q(t))-\eta_{\eps}(x+2\eps)].
	\end{align*}
	Using $\varphi_{\eps}$ as a test function and choosing $\eps>0$ smaller than $\min\{-x/2,(r-s)/2\}$, we obtain
	\begin{equation}
		\begin{split}
			&\int_{r-\eps}^{r+\eps}\eta_{\eps}(t)\int_{q(t)-\eps}^{0}\abs{u-v}\psi_{\eps}(x,t)dxdt-\int_{s-\eps}^{s+\eps}\eta_{\eps}(t)\int_{q(t)-\eps}^{0}\abs{u-v}\psi_{\eps}(x,t)dxdt \\
			\leq&\int_{s-\eps}^{r+\eps}\tau_{\eps}(t)\int_{q(t)-\eps}^{q(t)+\eps}\eta_{\eps}(x-q(t))\left\{-\dot{q}(t)\abs{u-v}+\operatorname{sgn}(u-v)[f(x,u)-f(x,v)]\right\}dxdt \\
			&-\int_{-3\eps}^{-\eps}\eta_{\eps}(x)\int_{s-\eps}^{r+\eps}\operatorname{sgn}(u-v)[f(x,u)-f(x,v)]\tau_{\eps}(t)dtdx.
		\end{split}
	\end{equation}
	Consider $t$ such that $u(q(t),t)\neq v(q(t),t)$. If there are no such times $t\in[0,T]$ then the first integral on the right-hand side vanishes anyway. Now, by the mean value property and our uniform bounds on $f\circ u,f\circ v$, we have that for some $\lambda\in(0,1)$:
	\[
	\dfrac{f(q(t),u(q(t),t))-f(q(t),v(q(t),t))}{u(q(t),t)-v(q(t),t)}=g(q(t))h^{\prime}(u+\lambda(v-u)),
	\]
	but since $h^{\prime}$ is monotone increasing,
	\[
	\dfrac{f(q(t),u(q(t),t))-f(q(t),v(q(t),t))}{u(q(t),t)-v(q(t),t)}<\dot{q}(t),
	\]
	and therefore
	\[
	-\dot{q}(t)\abs{u-v}(q(t),t)-\operatorname{sgn}(u-v)[f\circ u-f\circ v](q(t),t)\leq0.
	\]
	Since $\supnorm{f\circ u},\supnorm{f\circ v}<M$, we have that for small enough $\eps>0$, the integrands of both `boundary' integrals on the right-hand side are well-signed by our construction of $q$ and the interface condition at $0$. Hence, as $\eps\to0$, we have that for almost all $0<s<r<T$:
	\[
	\int_{q(r)}^{0}\abs{u-v}(x,r)dx\leq\int_{q(s)}^{0}\abs{u-v}(x,s)dx.
	\]
	Hence, letting $s\to0,r\to T$, we obtain the desired inequality.
\end{proof}
We remark that the levels $M$ generating the backward characteristic trajectories can be chosen arbitrarily close to $\max\{\supnorm{f\circ u},\supnorm{f\circ v}\}$ in the proof above, and hence the result also holds if $R(x)$ is defined in terms of the trajectories generated by the level $M=\max\{\supnorm{f\circ u},\supnorm{f\circ v}\}$. The proof of stability proceeds similarly in all other cases as well. Hence, flux-bounded entropy solutions of \eqref{claw} satisfying the interface conditions on $\partial G$ are unique when the flux satisfies the non-integrability conditions on \eqref{growth cond}. If $G^c=\mathbb{R}$, the argument proceeds by constructing maximal characteristics from both ends of a compact domain instead, and there are no interface conditions.

Finally, we remark that the conditions are well-defined even when $(a,b)=(-\infty,\infty)$, and the domain of dependence argument in the proof of Theorem~\ref{infty stability} merely has to take both $V_M^{\pm}$ into account.

\subsection{Infinite speed of propagation}\label{infspeed}
The non-integrability assumption(s) on \eqref{growth cond} ensuring finite speed of propagation are crucial in cases where the support $G^c$ of $g$ contains an unbounded interval. The following simple construction demonstrates that the assumption is sharp.

Consider the following flux of multiplicative form\footnote{We thank an anonymous referee for alerting us to this example of non-uniqueness.}
	\begin{equation}\label{fastflux}
		f(x,u)=\exp(-2x)u^2.
	\end{equation}
	Note that $f$ satisfies \eqref{LL} since $\log(u^2)=2\log(u)$. If $u_0(x)=\exp(x)$, then clearly $f\circ u_0(x)=1$, and hence $u(x,t)=u_0(x)$ is an entropy solution to the Cauchy problem \eqref{claw} with given flux and initial data $u_0(x)$. However, it is not the unique entropy solution, since $v(x,t)$ defined by
	\[
	v(x,t)=
	\begin{cases}
		e^x, & e^x\ge 2t,\\[1ex]
		\dfrac{e^{2x}}{2t}, & e^x<2t,
	\end{cases}
	\]
	is also an entropy solution with a shock wave travelling infinitely fast from $-\infty$ at the initial time, since it verifies the Rankine-Hugoniot condition along the curve of discontinuity $\gamma(t)=\log(2t)$, and the jump itself is entropic by convexity of $f$. More generally, for $\lambda>0$, the functions $v_{\lambda}(x,t)$ defined by
	\[
	v_{\lambda}(x,t)=
	\begin{cases}
		e^x, & e^x\ge2(t-\lambda),\\[1ex]
		\dfrac{e^{2x}}{2(t-\lambda)}, & e^x<2(t-\lambda),
	\end{cases}
	\]
	are all entropy solutions of the same Cauchy problem. We can truncate $f$ for large positive values of $x$ without affecting the preceding non-uniqueness result such that $f(x,u)$ is uniformly convex with respect to $u$. Hence, the non-integrability assumption is crucial for well-posedness even in the case of uniformly convex flux with Hamilton-Jacobi correspondence. Note that `compactly non-homogeneous' fluxes \cite{ConLawHJB} trivially satisfy this condition, but in general the assumption rules out fluxes like \eqref{fastflux}.

The non-integrability condition on $V_M^{\pm}$ is sharp, i.e. the illustrative example above has a generic counterpart. Suppose one of the conditions fails for a flux $f(x,u)=g(x)h(u)$; without loss of generality, suppose $(-\infty,b)\subseteq G^c$ is maximal, $g(x)>0$ for all $x\in(-\infty,b)$, and $V_M^{-}(\overline{x})=T<\infty$ for some $M>0$ and $\overline{x}\in(-\infty,b)$. Violations of other conditions can be handled similarly. By monotonicity of characteristic speeds, we have that $V_{M+1}^{-}(\overline{x})=T^{\prime}<T$. Define the curves $q_0,q_1$ as solutions of the following terminal value problems: for $k=0,1$ respectively,
\begin{align*}
	\dot{q}_k(t)&=g(q(t))h^{\prime}\circ h_+^{-1}((M+k)g(q(t))^{-1}), \\
	q_k(T)&=\overline{x}.
\end{align*}
The curves $q_0,q_1$ are well-defined on the respective intervals $(0,T)$ and $(T-T^{\prime},T)$, and approach $-\infty$ as $t\to0,T-T^{\prime}$ respectively. Now consider the Rankine-Hugoniot curve $\gamma(t)$ passing through $(\overline{x},T)$ and defined on some maximal interval of the form $(T^*,T)$ joining the positive stationary solutions corresponding to $M+1$ and $M$. Since the discontinuity is entropic, the curve $\gamma(t)$ cannot cross either $q_0$ or $q_1$ in its domain of definition. Hence, $0\leq T^*\leq T-T^{\prime}$, and for all $t\in(T-T^{\prime},T)$, we have that
\[
q_1(t)<\gamma(t)<q_0(t).
\]
Now define a function $u(x,t)$ on $\Omega_T=\R\times[0,T]$ as follows:
\[
u(x,t)=
\begin{cases}
	0&\text{ if }x\geq b, \\
	h_+^{-1}(Mg(x)^{-1})&\text{ if }t<T^{*}, \\
	h_+^{-1}(Mg(x)^{-1})&\text{ if }t>T^{*},b>x>\gamma(t), \\
	h_+^{-1}((M+1)g(x)^{-1})&\text{ otherwise}.
\end{cases}
\]
Clearly $u_0(x)=u(x,0)=h_{+}^{-1}(Mg(x)^{-1})$ for all $x\in(-\infty,b)$ and $u_0(x)=0$ for all $x\in(b,\infty)$. Furthermore, the curve of discontinuity $\gamma(t)$ satisfies the Rankine-Hugoniot condition, and $u(\gamma(t)-,t)>u(\gamma(t)+,t)$ for all $t\in(T^*,T)$, hence $u(x,t)$ is an entropy solution of the Cauchy problem \eqref{claw} with initial data
\begin{equation}\label{nonuni init data}
	u_0(x)=
	\begin{cases}
		0&\text{ if }x>b, \\
		h_{+}^{-1}(Mg(x)^{-1})&\text{ if }x<b.
	\end{cases}
\end{equation}
Note that if $b=+\infty$, then $\{x>b\}$ is simply the empty set. However, the Cauchy problem \eqref{claw} with initial data \eqref{nonuni init data} trivially has another entropy solution, namely the stationary one $u(x,t)=u_0(x)$. Thus, the Cauchy problem is generally ill-posed for this flux. Since the construction is generic, we may conclude that the non-integrability condition on maximal unbounded intervals of $G^c$ is sharp with respect to the well-posedness of \eqref{claw}.

\section{Explicit examples}
We solve some more initial value problems explicitly, and demonstrate that our sufficient conditions for $\L{\infty}$ blow-up is sharp.

\subsection{Heterogeneity with vanishing derivative}\label{non int}
To see that the derivative condition in Theorem~\ref{vander} is sharp, consider the flux
\begin{align}\label{double convex}
	f(x,u)=x^2u^2.
\end{align}
The characteristic ODEs for \eqref{claw} reduce in this case to
\begin{align}
	\dot{q}(t)&=2q(t)^2p(t),\label{doubly convex char q} \\
	\dot{p}(t)&=-2q(t)p(t)^2.\label{doubly convex char p}
\end{align}
If $q(0)=0$ or $p(0)=0$, then $q(t)=q(0)$ and $p(t)=p(0)$. This equation \eqref{doubly convex char p} is no longer autonomous; the system is now genuinely coupled. However, since $f$ is invariant along trajectories, we know that $q(t)p(t)=q(0)p(0)$. Hence, the equations reduce to
\begin{align}
	\dot{q}(t)&=2q(0)p(0)q(t),\label{explicit doubly convex char q} \\
	\dot{p}(t)&=-2q(0)p(0)p(t).\label{explicit doubly convex char p}
\end{align}
If we denote $C=2q(0)p(0)$, then the system of equations \eqref{explicit doubly convex char q},\eqref{explicit doubly convex char p} have explicit solution
\begin{align}
	q(t)&=q(0)\exp(Ct),\label{doubcon sol q} \\
	p(t)&=p(0)\exp(-Ct),\label{doubcon sol p}
\end{align}
and thus neither the characteristic trajectory nor the value can blow up in finite time. Note that $f$ as defined in \eqref{double convex} does \textit{not} satisfy our structural assumption, however, and thus stationary solutions are not integrable. This is easily seen as follows. Let $f=1$, and consider the positive stationary solution $u$ associated with it. Then (to rephrase) 
\[
u(x,t)=\sqrt{\dfrac{1}{x^2}}=\dfrac{1}{\abs{x}},
\]
which is clearly not integrable. Hence, this also shows that the exponent in our structural constraint is sharp. If $u_0\in \L{\infty}$, however, we can avoid this issue since we have explicit control on the $\L{\infty}$ norm. At the front tracking level, we choose our $\delta$-approximate piecewise stationary solution to be zero in a small neighbourhood of the origin, ensuring that we remain in the $\L{\infty}$ and hence $\Lloc{1}$ class.

\subsection{Ill-posedness}\label{illp}
If we relax the constraint that our flux be of multiplicative form, the blow-up can be even more pathological. Consider the following spatially heterogeneous flux for \eqref{claw}
\begin{equation}\label{coercive flux}
	f(x,u)=xu^2+u^4.
\end{equation}
The equations \eqref{char ode} then reduce to
\begin{align}
	\dot{q}(t)&=2q(t)p(t)+4p(t)^3, \label{char ode coercive flux q} \\
	\dot{p}(t)&=-p(t)^2. \label{char ode coercive flux p}
\end{align}
Note that \eqref{char ode coercive flux p} is still an autonomous Riccati equation that may blow up in finite time and is explicitly given by $p(t)\equiv0$ if $p(0)=0$, and
\[
p(t)=\dfrac{1}{t+p(0)^{-1}}
\]
if $p(0)\neq0$. Consider the constant initial datum $u_0 \equiv-1$. By the method of characteristics, a smooth solution exists locally in time. Given the explicit form of $p(t)$, we can compute $q(t)$ for any value of $q(0)$ by conservation of $f$ along trajectories \eqref{char ode coercive flux q}-\eqref{char ode coercive flux p}. In particular, we have that
\[
q(t)p(t)^2+p(t)^4=q(0)p(0)^2+p(0)^4,
\]
and hence, for $p(0)\neq0$,
\begin{equation}\label{q formula 2}
	\begin{split}
		q(t)&=p(t)^{-2}\left(q(0)p(0)^2+p(0)^4-p(t)^4\right) \\
		&=\dfrac{q(0)+1}{p(t)^2}-p(t)^2.
	\end{split}
\end{equation}
Since $p(0)=-1$ for all characteristics, we can write \eqref{q formula 2} as
\[
q(t)=\left(1-t\right)^2\left(q(0)+1\right)-\dfrac{1}{\left(1-t\right)^2},
\]
and thus $q(t)$ blows up as $t\to1$ as well. Now, it is clear that for $t<1$, no two characteristics can intersect, and that the value of $u(x,t)$ for any fixed time level is constant in $x$. Hence, the Cauchy problem corresponding to this flux is ill-posed globally in time -- there is no way to continue the solution beyond $t=1$, even if we look for merely locally integrable functions.

As with the construction of non-unique solutions, this ill-posedness result can be generalised to arbitrary spatial dimensions by a simple extension. In particular, the following Cauchy problem is globally ill-posed, for any arbitrary choice of functions $f_i$:
\begin{equation}\label{multidbad}
	\begin{split}
		u_t+(x_1u^2+u^4)_{x_1}+\sum_{i=2}^{n}f_i(u)_{x_i}&=0, \\
		u(x,0)&=-1.
	\end{split}
\end{equation}
The ill-posedness of \eqref{multidbad} is a trivial consequence of the result for \eqref{coercive flux}.

\section{Acknowledgements}
SSG and PV would like to thank the Department of Atomic Energy, Government of India, for their support under project no. 12-R\&D-TFR-5.01-0520.

\bibliographystyle{abbrv}
\bibliography{citations}
\end{document}